\documentclass{compositio}
    
\usepackage{latexsym}
\usepackage{amsfonts}
\usepackage{amssymb}
\usepackage{amsmath}
\usepackage{verbatim}
\usepackage{amsthm}
\def\ints{{\mathbb Z}}
\def\rats{{\mathbb Q}}

\def\reals{{\mathbb R}}

\def\proj{{\mathbb P}}

\def\FF{{\mathbb F}}

\def\qed{\hfill $\Box$}

\def\ord{{\text{ord}}}

\DeclareMathOperator{\eff}{eff}

\def\Frac{{\text{Frac}}}

\DeclareMathOperator{\Aut}{Aut}

\def\Gal{{\text{Gal}}}

\def\Spec{{\mbox{Spec }}}

\def\mc#1{\mathcal{#1}}

\def\ol#1{\overline{#1}}

\def\matrix#1#2#3#4{
    \left( \begin{array}{cc} #1&#2 \\ #3&#4 \end{array} \right)}

\newtheorem{theorem}{Theorem}[section]
\newtheorem{prop}[theorem]{Proposition}
\newtheorem{lemma}[theorem]{Lemma}
\newtheorem{corollary}[theorem]{Corollary}
\newtheorem{predefinition}[theorem]{Definition}
\newenvironment{definition}{\begin{predefinition}\rm}{\end{predefinition}}
\newtheorem{preremark}[theorem]{Remark}
\newenvironment{remark}{\begin{preremark}\rm}{\end{preremark}}
\newtheorem{preconstruction}[theorem]{Construction}
\newenvironment{construction}{\begin{preconstruction}\rm}{\end{preconstruction}}
\newtheorem{prenotation}[theorem]{Notation}
\newenvironment{notation}{\begin{prenotation}\rm}{\end{prenotation}}
\newtheorem{preexample}[theorem]{Example}
\newenvironment{example}{\begin{preexample}\rm}{\end{preexample}}
\newtheorem{preclaim}[theorem]{Claim}

\newtheorem{prequestion}[theorem]{Question}
\newenvironment{question}{\begin{prequestion}\rm}{\end{prequestion}}

\begin{document}
\title{Vanishing cycles and wild monodromy}

\author{Andrew Obus}
\email{obus@math.columbia.edu}
\address{Department of Mathematics\\Columbia University\\MC 4403\\2990
Broadway\\New York, NY 10027}
\thanks{Supported by an NDSEG Graduate Research Fellowship and an NSF
Mathematical Sciences Postdoctoral Research
Fellowship.}

\numberwithin{equation}{section}

\begin{abstract}
Let $K$ be a complete discrete valuation field of mixed characteristic $(0,p)$
with algebraically closed residue field,
and let $f: Y \to \proj^1$ be a three-point $G$-cover defined over $K$, where
$G$ has a cyclic $p$-Sylow subgroup $P$.  We
examine the stable model of $f$, in particular, the minimal extension $K^{st}/K$
such that the stable model is defined
over $K^{st}$.  Our main result is that, if $|P| = p^n$ and the center of $G$ has prime-to-$p$ order, 
then the $p$-Sylow subgroup of
$\Gal(K^{st}/K)$ has exponent dividing $p^{n-1}$.  This extends work of Raynaud
in the case that $|P| = p$.
\end{abstract}

\maketitle

\section{Introduction}\label{Sintro}
This paper investigates the stable reduction of three-point covers over complete
discrete valuation
fields of mixed characteristic.  In particular, we examine the minimal extension
of a field of definition of such a
cover that is necessary to obtain the stable model.

Let $G$ be a finite group.  Let $f:Y \to X$ be a branched $G$-Galois cover of
curves defined over a complete discretely valued field $K$ of 
characteristic $0$ with valuation ring $R$ whose residue field $k$ is
algebraically closed of characteristic $p$.  Suppose $X$ has a smooth model over
$R$.
Then, under mild hypotheses, there is a finite extension $K^{st}/K$ with
valuation ring $R^{st}$ such
that there is a minimal \emph{stable model} $f_{R^{st}}$ for $f \times_K K^{st}$
whose 
special fiber $\ol{f}: \ol{Y} \to \ol{X}$ (called the \emph{stable reduction})
is reduced with only nodal singularities.
The group $G$ acts on $\ol{Y}$.  See \S\ref{Sstable} for more details.

The proofs of the existence of stable reduction are non-constructive, and the
question of determining the minimal extension 
$K^{st}/K$ over which the stable model can be defined is far from being answered
in most cases.  In the case that the
branch points of $f$ are defined over $K$, the extension $K^{st}/K$ is Galois,
and we focus on the (unique) $p$-Sylow
subgroup $\Gamma_w$ of $\Gal(K^{st}/K)$, called the \emph{wild monodromy group}
of $f$.  This group acts faithfully on
the stable reduction $\ol{f}$.

The wild monodromy group has been studied 
by Lehr and Matignon in the case of a $\ints/p$-cover of $\proj^1$ with any
number of branch points (\cite{LM:wm}), 
and by Raynaud in the case of a three-point $G$-cover (i.e., a cover of
$\proj^1$ branched exactly at $0, 1$, and $\infty$) 
such that $p$ exactly divides $|G|$ (\cite{Ra:sp}).  Note that, in both cases,
$p$ exactly divides the order of the Galois 
group of $f$.  The methods of \cite{LM:wm} exploit the availability of explicit
equations coming from having a Kummer 
cover, whereas the methods of \cite{Ra:sp} involve understanding combinatorial
aspects of $\ol{f}$.
In particular, Raynaud proves a \emph{vanishing cycles formula} (\cite[3.4.2
(5)]{Ra:sp}) which places restrictions on
how complicated $\ol{f}$ can be, which in turn places bounds on the wild
monodromy group. 
Also, \cite{Ra:sp} relates the wild monodromy group to the question of good
reduction of $f$.  

This paper builds on the ideas of \cite{Ra:sp}.  We place bounds on the wild
monodromy group $\Gamma_w$ of $f$, when $f$ is 
a three-point $G$-cover such that $G$ has a \emph{cyclic} $p$-Sylow subgroup of
\emph{arbitrary} order.  Specifically, our main result is the following:

\begin{theorem}\label{Tmain}
Suppose $G$ is a finite group with nontrivial cyclic $p$-Sylow subgroup $P$ of order
$p^n$.  Assume that $p$ does not divide the order 
of the center of $G$.   If $f: Y \to X = \proj^1$ is a three-point $G$-cover 
defined over $K$, then the wild monodromy group $\Gamma_w$ has exponent dividing
$p^{n-1}$.    
\end{theorem}

\begin{remark}\label{Rmain}
\begin{description}
\item{(i)} If $n = 1$, Theorem \ref{Tmain} says that $\Gamma_w$ is trivial, which is
the result that appeared as \cite[Th\'{e}or\`{e}me 4.2.10 (1)]{Ra:sp} (for $Y$ having genus $\geq 2$).  
\item{(ii)} If the branching indices of $f$ are prime to $p$ and there are no ``new" \'{e}tale tails (see \S\ref{Sstable}), then $\Gamma_w$ is trivial, 
regardless of $n$ (Proposition \ref{Ptrivwildmon}).
\item{(iii)} We are not yet sure if, in fact, there exist examples where $n \geq 2$, where $p$ does not divide the order of the center of $G$, 
and where $\Gamma_w$ has exponent $p^{n-1}$.  Indeed, it is difficult to
write down examples where there is any wild monodromy at all.  We write down such an example in Appendix \ref{Awildexample} 
where $p=5$, $n=3$, and $\Gamma_w \supseteq \ints/5$.
\item{(iv)} Using Raynaud's relation between the wild monodromy group and good reduction, we exhibit a family of
three-point $G$-covers with potentially good reduction to characteristic $p$, where $G$ has arbitrarily large cyclic $p$-Sylow subgroup.  
Specifically, $G \cong PGL_3(q)$, with $p^n \, |  \,q^2 + q + 1$.  See Example \ref{Xgoodreduction}.
\end{description}
\end{remark}

Several difficulties present themselves when we allow $G$ to have a $p$-Sylow
subgroup of order greater than $p$.  The
first is that Raynaud's vanishing cycles formula is proven only in the case
where $p$ exactly divides $|G|$ (in fact, it
is not immediately obvious what the generalization should be for arbitrary $G$).
 Our Theorem \ref{Tvancycles} extends
this formula to the case where $G$ has a cyclic $p$-Sylow subgroup.  The proof
uses invariants of \emph{deformation data}, 
which were used by Wewers in \cite{We:br} to give an alternate proof of
Raynaud's vanishing cycles formula.  Wewers 
associated deformation data to irreducible components of $\ol{Y}$ on which $G$
acts with inertia group of order $p$.  We 
extend this to the case where $G$ acts with cyclic inertia of any order, and
introduce new \emph{effective invariants}
for these deformation data.  

A more serious difficulty is that when $p$ divides the order of $|G|$ more than
once, then $\ol{f}: \ol{Y} \to \ol{X}$
can have \emph{inseparable tails}, i.e., irreducible components $\ol{W}$ of
$\ol{X}$ that intersect the rest of $\ol{X}$ 
at one point such that $G$ acts with nontrivial inertia above $\ol{W}$.  These
tails are not present when $p$ exactly
divides $|G|$, and they do not appear in Theorem \ref{Tvancycles}.  We prove a
\emph{generalized vanishing cycles
formula} (Proposition \ref{Pgenvancycles}) that takes these tails into account. 
This requires introducing
what we call \emph{truncated effective invariants} of deformation data.  Using
Proposition \ref{Pgenvancycles}, we are 
able to prove Theorem \ref{Tmain}.

There are proofs of Theorem \ref{Tvancycles} and Proposition \ref{Pgenvancycles}
that do not use deformation data (see
\cite[\S3.1]{Ob:th}), but the proof we give here is particularly nice, and the
material on deformation data that we
develop here will be used in the subsequent papers \cite{Ob:fm1} and \cite{Ob:fm2}.

\subsection{Section-by-section summary and walkthrough}
In \S2, we give preliminary results about group theory, stable reduction, and
ramification.  Many of these results are already known.  
In \S3 we generalize the construction of deformation data given in \cite{He:ht}
and used in \cite{We:br}, and use these 
deformation data to prove the aforementioned vanishing cycles formulas. For a
three-point $G$-cover $f:Y \to X$, 
results limiting the number and type of tails of the stable reduction $\ol{f}:
\ol{Y} \to \ol{X}$ are given in \S4.
The main theorem is proved in \S5, where we use arguments similar to, but more
complicated than, those of
\cite[\S4.2]{Ra:sp} to obtain our restrictions on the wild monodromy of $f$. 
The connection to good reduction is
discussed at the end of \S5.

\subsection{Notation and conventions}\label{Snotations}
The following notations will be used throughout the paper:  The letter $p$
always represents a prime number.
If $G$ is a group, and $H$ a subgroup, we write $H \leq G$.  
We denote by $N_G(H)$ the normalizer of $H$ in $G$ and by $Z_G(H)$ the
centralizer of $H$ in $G$.  The order of $G$ is
written $|G|$.  If $G$ has a cyclic $p$-Sylow subgroup $P$, and $p$ is
understood, we write $m_G = |N_G(P)/Z_G(P)|$.  

If $K$ is a field, $\ol{K}$ is its algebraic closure.  We write $G_K$ for the
absolute Galois group of $K$.  If $H \leq
G_K$, we write $\ol{K}^H$ for the fixed field of $H$ in $\ol{K}$.  Similarly, if
$\Gamma$ is a group of automorphisms of a
ring $A$, we write $A^{\Gamma}$ for the fixed ring under $\Gamma$.
If $K$ is discretely valued, then $K^{ur}$ is the \emph{completion} of the
maximal unramified algebraic extension of $K$.  

If $x$ is a scheme-theoretic point of a scheme $X$, then $\mc{O}_{X,x}$ is the
local ring of $x$ on $X$.  
If $R$ is any local ring, then $\hat R$ is the completion of $R$ with respect to
its maximal ideal. 
If $R$ is a discrete valuation ring with fraction field $K$ of characteristic 0
and residue field $k$ of
characteristic $p$, we normalize the valuation $v$ on $R$ so that $v(p) = 1$.  

A \emph{branched cover} $f: Y \to X$ of smooth proper curves is a finite,
surjective, generically \'{e}tale morphism. 
All branched covers are assumed to be geometrically connected.
If $f$ is of degree $d$ and $G$ is a finite group of order $d$ with $G \cong
\Aut(Y/X)$, then $f$ is called a
\emph{Galois cover with (Galois) group $G$}.  If we choose an isomorphism $i: G
\to \Aut(Y/X)$, then the datum $(f, i)$
is called a \emph{$G$-Galois cover} (or just a \emph{$G$-cover}, for short).  We
will usually suppress the isomorphism
$i$, and speak of $f$ as a $G$-cover. 
 
The \emph{ramification index} of a
point $y \in Y$ such that $f(y) = x$ is the ramification index of the extension
of complete local rings $\hat{\mc{O}}_{X, x}
\to \hat{\mc{O}}_{Y, y}$.  If $f$ is Galois, then the \emph{branching index} of
a closed point $x \in X$ is the ramification
index of any point $y$ in the fiber of $f$ over $x$.  If $x \in X$ (resp.\ $y
\in Y$) has branching index (resp.\
ramification index) greater than 1, then it is called a \emph{branch point}
(resp.\ \emph{ramification point}).

If $X$ is a smooth curve over a complete discrete valuation field $K$ with
valuation ring $R$, then a \emph{semistable}
model for $X$ is a relative curve $X_R \to \Spec R$ with $X_R \times_R K \cong
X$ and semistable special fiber (i.e.,
the special fiber is reduced with only ordinary double points for
singularities).

For any real number $r$, $\lfloor r \rfloor$ is the greatest integer less than
or equal to $r$. 
Also, $\langle r \rangle := r - \lfloor r \rfloor$.

\section{Background Material}\label{Sbackground}

\subsection{Finite groups with cyclic $p$-Sylow subgroups}\label{Sgroups}
In this section, we prove structure theorems about finite groups with
\emph{cyclic} $p$-Sylow subgroups.   
Throughout \S\ref{Sgroups}, $G$ is a finite group with a cyclic $p$-Sylow
subgroup $P$ of order $p^n$.  Recall that $m_G = |N_G(P)/Z_G(P)|$.

\begin{lemma}\label{Lautaction}
Let $Q \leq P$ have order $p$.  If $g \in N_G(P)$ acts trivially on $Q$ by
conjugation, it acts trivially on $P$.
Thus $N_G(P)/Z_G(P) \hookrightarrow \Aut(Q)$, so $m_G|(p-1)$.
\end{lemma}

\begin{proof}(cf. \cite[Remarque 3.1.8]{Ra:sp})
We know $\Aut(P) \cong (\ints/p^n)^{\times}$, which has order $p^{n-1}(p-1)$,
with a unique maximal prime-to-$p$ 
subgroup $C$ of order $p-1$.
Let $g \in N_G(P)$, and suppose that the image $\ol{g}$ of $g$ in $N_G(P)/Z_G(P)
\subseteq \Aut(P)$ acts trivially on $Q$.
Since $$(\ints/p^n)^{\times} \cong \Aut(P) \twoheadrightarrow \Aut(Q) \cong
(\ints/p)^{\times}$$ has $p$-group kernel
we know that $\ol{g}$ has $p$-power order.
If $\ol{g}$ is not trivial, then $g \notin P$, and the subgroup $\langle g, P
\rangle$ of $G$ has a non-cyclic 
$p$-Sylow subgroup.   
This is impossible, so $\ol{g}$ is trivial, and $g$ acts trivially on $P$.
\end{proof}

We state a theorem of Burnside:

\begin{lemma}[(\cite{Za:tg}, Theorem 4, p.\ 169)]\label{Lburnside}
Let $\Gamma$ be a finite group, with a $p$-Sylow subgroup $\Pi$.  Then, if
$N_{\Gamma}(\Pi) = Z_{\Gamma}(\Pi)$, the
group $\Gamma$ can be written as an extension
$$1 \to \Sigma \to \Gamma \to \Pi' \to 1,$$ where $\Pi \leq \Gamma$ maps
isomorphically onto $\Pi'$. 
\end{lemma}

\begin{lemma}\label{Lzpzm}
Suppose that a finite group $G'$ has  
a normal subgroup $Q$ of order $p$ contained in a cyclic $p$-Sylow subgroup $P$,
and
no nontrivial normal subgroups of prime-to-$p$ order.  
Then $G' \cong P \rtimes \ints/m_{G'}$.  In particular, $G'$ is solvable.
\end{lemma}

\begin{proof}
Consider the centralizer $C := Z_{G'}(Q)$.  
Now, $P$ is clearly a $p$-Sylow subgroup of $C$.  
By the definition of $C$, any element in $N_C(P)$ acts trivially on $Q$ by
conjugation.  So $N_C(P)$ acts trivially on 
$P$ as well.  Thus $N_C(P) = Z_C(P)$.

Since $N_C(P) = Z_C(P)$, Lemma \ref{Lburnside} shows that $C$ can be written
as an extension $$1 \to S \to C \to \ol{P} \to 1,$$ where $P \subseteq C$ maps
isomorphically onto $\ol{P}$. 
The group $S$, being the maximal normal prime-to-$p$ subgroup of $C$, is
characteristic
in $C$.  Since $C$ is normal in $G'$ (it is the centralizer of the normal
subgroup $Q$), $S$ is normal in $G'$.
But by assumption, $G'$ has no nontrivial normal subgroups of prime-to-$p$
order, so $S$ is
trivial and $C = P$.  Again, since $C$ is normal in $G'$, then $G'$ is of the
form $P \rtimes T$, 
where $T$ is prime to $p$, by the Schur-Zassenhaus theorem.
The conjugation action of $T$ on $P$ must be faithful, since if there were
a kernel, the kernel would be a nontrivial prime-to-$p$ normal subgroup of $G'$,
contradicting the assumption that $G'$ has
none.  Since the subgroup of $\Aut(P)$ induced by this action is cyclic of order
$m_{G'}$, we have
$T \cong \ints/m_{G'}$. 
\end{proof}

\begin{corollary}\label{Cnormalp}
\begin{description}
\item{(i)} If $G$ has a normal subgroup of order $p$, then there exists a normal
prime-to-$p$ subgroup $N < G$ such that 
$G/N \cong \ints/p^n \rtimes \ints/m_G$.
\item{(ii)} If $G$ has a central subgroup of order $p$, then there exists a
normal prime-to-$p$ subgroup $N < G$ such that 
$G/N \cong \ints/p^n$.  In particular, $m_G = 1$.
\end{description}
\end{corollary}
\begin{proof}
In both cases, let $N$ be the maximal normal prime-to-$p$ subgroup of $G$.  
Then $G/N$ still has a normal subgroup of order $p$, and no nontrivial normal
subgroups of prime-to-$p$ order. 
Part (i) then follows from Lemma \ref{Lzpzm}.  If $G$ has a central subgroup of
order $p$, then by Lemma
\ref{Lautaction}, we have $m_G = 1$.  Part (ii) follows.
\end{proof}

\subsection{Basic facts about (wild) ramification}\label{Sramification}
We state here some facts from \cite[IV]{Se:lf} and derive some consequences.
Let $K$ be a complete discrete valuation field with algebraically closed 
residue field $k$ of characteristic $p > 0$.  If $L/K$ is a finite Galois
extension of fields with Galois group $G$, then 
$L$ is also a complete discrete valuation field with residue field $k$.  Here
$G$ is
of the form $P \rtimes \ints/m$, where
$P$ is a $p$-group and $m$ is prime to $p$.  The group $G$ has a filtration $G 
= G_0
\supseteq G_i$ ($i \in \reals_{\geq 0}$) for the lower numbering, and $G
\supseteq G^i$ for the upper numbering 
($i \in \reals_{\geq 0}$).  If $i \leq j$, then $G_i \supseteq G_j$ and $G^i
\supseteq G^j$ 
(see \cite[IV, \S1, \S3]{Se:lf}).   
The subgroup $G_i$ (resp.\ $G^i$) is known as the \emph{$i$th higher
ramification group for
the lower numbering (resp.\ the upper 
numbering)}.  One knows that $G_0 = G^0 = G$, and that for sufficiently small
$\epsilon > 0$, 
$G_{\epsilon} = G^{\epsilon} = P$.  For sufficiently large $i$, $G_i = G^i =
\{id\}$.    
Any $i$ such that $G^i \supsetneq G^{i + \epsilon}$ for all $\epsilon > 0$ is
called an \emph{upper jump} of the extension $L/K$.  Likewise, if $G_i
\supsetneq G_{i+\epsilon}$, then $i$ is called a
\emph{lower jump} of $L/K$.  If $i$ is a lower (resp.\ upper) jump and $i > 0$,
then $G^i/G^{i + \epsilon}$ (resp.\ $G_i/G_{i + \epsilon}$)
is an elementary abelian $p$-group.  The lower jumps are all integers.
The greatest upper jump (i.e., the greatest $i$ such that $G^i \neq \{id\}$) is
called the
\emph{conductor} of higher ramification of $L/K$.  The upper numbering is
invariant under
quotients (\cite[IV, Proposition 14]{Se:lf}).  That is, if $H \leq G$ is normal,
and $M = L^H$, then the $i$th higher
ramification group for the upper numbering for $M/K$ is $G^i/(G^i \cap H)$.

\begin{lemma}\label{Lhassearf}
If $P$ is abelian, then all upper jumps (in particular, the conductor of higher
ramification) are in 
$\frac{1}{m}\ints$.
\end{lemma}

\begin{proof}
Let $L_0 \subset L$ be the fixed field of $L$ under $P$.
By the Hasse-Arf theorem (\cite[V, Theorem 1]{Se:lf}), the upper jumps for the
$P$-extension $L/L_0$ are integers.
By Herbrand's formula (\cite[IV, \S3]{Se:lf}), the upper jumps for $L/K$ are
$\frac{1}{m}$ times those for
$L/L_0$.  The lemma follows.
\end{proof}

The following lemma will be useful in \S\ref{Swild}.  In the case $G \cong
\ints/p \rtimes \ints/m$, it is essentially
\cite[Propositions 1.1.4, 1.1.5]{Ra:sp}.

\begin{lemma}\label{Lcompositum}
Fix an algebraic closure $\ol{K}$ of $K$.  
Suppose $L/K$ is a finite $G$-Galois extension with conductor $\sigma$, and
$K'/K$ is a 
$\ints/p$-Galois extension with conductor $\tau \leq \sigma$.  
Assume we can embed $L/K$ and $K'/K$ into $\ol{K}/K$ so that they are linearly
disjoint, and write $L'$ for the
compositum $LK'$ in $\ol{K}$.  Write $\rho$ for the conductor of $L'/K'$.  Then
$\rho-\tau = p(\sigma-\tau)$.
\end{lemma}

\begin{proof}
Since $L$ and $K'$ are linearly disjoint, $G'' := \Gal(L'/K) \cong G \times
\ints/p$.  Also, write $G \cong G'
:= \Gal(L'/K')$.  Since the upper numbering is invariant
under quotients, the conductor of $L'/K$ is equal to $\max(\sigma, \tau) =
\sigma$.  
Let $H := \Gal(L'/L) < \Gal(L'/K)$, 
and let $i$ be the greatest integer such that $G''_i \supseteq H$ but $G''_{i+1}
\not
\supseteq H$.  Clearly $H \cong \ints/p$.
Since the lower numbering is preserved under subgroups, we have that $G'_j =
G''_j$ for $j > i$.
Furthermore, $G''_j = G'_j \times H$ for $j \leq i$.  If $\phi'$ and $\phi''$
are the respective Herbrand functions of
$L'/K'$, $L'/K$ (\cite[IV, \S3]{Se:lf}), then $\phi'(i) = \phi''(i) = \tau$.  By
the definition of the upper numbering,
for $j > i$, we have $\phi'(j) - \phi'(i) = p(\phi''(j) - \phi''(i))$ for $j >
i$.  Taking $j$ to be maximal such that $G'_j$ is nontrivial,
we obtain $\rho - \tau = p(\sigma-\tau)$.
\end{proof}  

If $A, B$ are the valuation rings
of $K, L$, respectively, sometimes we will refer to the conductor or
higher ramification groups of the extension $B/A$.

\subsubsection{Smooth Curves.}\label{Ssmooth}
Let $f:Y \to X$ be a branched cover of smooth, proper, integral curves over $k$.
The Hurwitz formula (\cite[IV \S2]{Ha:ag}) states that 
$$2g_Y - 2 = (\deg f)(2g_X - 2) + |\Delta|,$$ where $\Delta$ is the ramification
divisor and $|\Delta|$ is its degree 
(recall that $\Delta = \sum_{y \in Y} d_y y$, where $d_y$ is the length of
$\Omega_{Y/X}$ at $y$). 
For each point $y \in Y$ with image $x \in X$, the degree of $\Delta$ at $y$ can
be related to the higher
ramification filtrations of $\Frac(\hat{\mc{O}}_{Y,
y})/\Frac(\hat{\mc{O}}_{X,x})$ (\cite[IV, Proposition 4]{Se:lf}).  
In particular, if the ramification index $e_y$ of $y$ is prime to $p$, then the
degree of $\Delta$ is $e_y - 1$.

In particular, suppose the Galois group $G$ of $\hat{\mc{O}}_{Y,
y}/\hat{\mc{O}}_{X, x}$ is isomorphic to $P \rtimes
\ints/m$ with $P$ \emph{cyclic} of order $p^n$.  Then all subgroups of $P$ must
occur as higher ramification groups (the 
subquotients of the higher ramification filtration having exponent $p$).
For $1 \leq i \leq n$, write $u_i$ (resp.\ $j_i$) for the upper (resp.\
lower) jump such that $G^{u_i}$ (resp.\ $G_{j_i}$) is isomorphic to
$\ints/p^{n-i+1}$.  Write $u_0 = j_0 = 0$.
Then $0 = u_0 < u_1 < \cdots < u_n$ and $0 = j_0 < j_1 < \cdots < j_n$.  We will
sometimes call $j_i$ (resp.\ $u_i$) the
$i$th \emph{lower jump} (resp.\ \emph{upper jump}) of the extension
$\hat{\mc{O}}_{Y, y}/\hat{\mc{O}}_{X, x}$.
Let $|\Delta_y|$ be the degree of $\Delta$ at $y$.

\begin{lemma}\label{Ljumpsdifferent}
\begin{description}
\item{(i)} In terms of the lower jumps, we have $$|\Delta_y|= p^nm - 1 +
\sum_{i=1}^n j_ip^{n-i}(p-1) = 
p^nm-1 + \sum_{i=1}^n (p^{n-i+1} - 1)(j_i - j_{i-1}).$$
\item{(ii)} In terms of the upper jumps, we have $$|\Delta_y| = p^nm - 1 +
\sum_{i=1}^n mp^{i-1}(p^{n-i+1} - 1)(u_i -
u_{i-1}).$$
\end{description}
\end{lemma}

\begin{proof}
By \cite[IV \S2]{Ha:ag}, $|\Delta_y|$ is equal to the valuation of the different
of the extension 
$\hat{\mc{O}}_{Y, y}/\hat{\mc{O}}_{X, x}$, where a uniformizer of
$\hat{\mc{O}}_{Y, y}$ is given valuation $1$.  By
\cite[IV, Proposition 4]{Se:lf}, this different is equal to 
$\sum_{r=0}^{\infty} (|G_r| - 1).$  Now it is a straightforward exercise to show
that (i) holds.  Part (ii) follows from
part (i) by Herbrand's formula (essentially, the definition of the upper
numbering).
\end{proof}

\begin{remark}\label{Rlower2conductor}
In the above context, it follows from Herbrand's formula that the conductor
$u_n$ is equal to 
$\frac{1}{m} \sum_{i=1}^n \frac{j_i - j_{i-1}}{p^{i-1}}$, which can also be
written as 
$\left(\sum_{i=1}^{n-1} \frac{p-1}{p^im}j_i \right) + \frac{1}{p^nm}j_n$.
\end{remark}

\subsection{Stable reduction}\label{Sstable}
We now introduce some notation that will be used for the remainder of 
\S\ref{Sbackground}.
Let $X/K$ be a smooth, proper, geometrically integral curve of genus $g_X$, 
where $K$ is a characteristic zero complete discretely
valued field, with algebraically closed residue field $k$ of characteristic
$p>0$ (e.g., $K =
\rats_p^{ur}$).  Let $R$ be the valuation ring of $K$.  Write $v$ for the
valuation on $R$.  
We normalize by setting $v(p) = 1$.  

For the rest of this section, assume that $X$ has a \emph{smooth} model $X_R$
over $R$.
Let $f: Y \to X$ be a $G$-Galois cover defined over $K$, with $G$ any finite
group, such that the branch points of $f$ are defined over $K$ and their
specializations 
do not collide on the special fiber of $X_R$.  Assume that
$2g_X - 2 + r \geq 1$, where $r$ is the number of branch points of $f$.  By a
theorem of Deligne and Mumford (\cite[Corollary 2.7]{DM:ir}), combined with work
of Raynaud (\cite{Ra:pg}, \cite{Ra:sp}) and Liu (\cite{Li:sr}), there is a
minimal finite extension $K^{st}/K$ 
with ring of integers $R^{st}$, and a unique model $Y^{st}$ of $Y_{K^{st}}$
(called the \emph{stable model}) such that

\begin{itemize}
\item The special fiber $\ol{Y}$ of $Y^{st}$ is semistable (i.e., it is reduced,
and has only nodes for singularities). 
\item The ramification points of $f_{K^{st}} = f \times_K K^{st}$ specialize to
\emph{distinct} smooth points of $\ol{Y}$.
\item Any genus zero irreducible component of $\ol{Y}$ contains at least three
marked points (i.e., ramification points or points of intersection with the rest
of $\ol{Y}$).
\end{itemize}
Since the stable model is unique, it is acted upon by $G$, and we set $X^{st} =
Y^{st}/G$.  Then
$X^{st}$ can be naturally identified with a blowup of $X \times_R R^{st}$
centered at closed points.
Furthermore, the nodes of $\ol{Y}$ lie above nodes of the special fiber $\ol{X}$
of $X^{st}$.

The map $f^{st}: Y^{st} \to X^{st}$ is called the \emph{stable model} of $f$ and
the field $K^{st}$ is called the minimal field of definition of the stable model
of $f$. 
Note that our definition of the stable model is the definition used in
\cite{We:br}.  This differs from the definition
in \cite{Ra:sp} in that \cite{Ra:sp} allows the ramification points to coalesce
on the
special fiber.  If we are working over a finite
extension $K'/K^{st}$ with ring of integers $R'$, we will sometimes abuse
language and call 
$f^{st} \times_{R^{st}} R'$ the stable model of $f$.  

For each $\sigma \in G_K$, $\sigma$ acts on $\ol{Y}$ and this action commutes
with $G$.  
Let $\Gamma^{st} \leq G_K$ consist of those $\sigma \in G_K$ such that $\sigma$
acts trivially on $\ol{Y}$.

\begin{prop}\label{Pstablecutout}
The extension $K^{st}/K$ is the extension cut out by $\Gamma^{st} \leq G_K$.  In
other words, $\Gamma^{st} =
G_{K^{st}}$.
\end{prop}

\begin{proof}
Choose $\gamma \in G_{K^{st}}$.  By Hensel's lemma, each smooth point $\ol{y}$
of $\ol{Y}$ is the specialization of 
a $K^{st}$-rational point $y$ of $Y^{st}$.  Since $\gamma$ fixes $y$, it fixes
$\ol{y}$.  Since the smooth points of $\ol{Y}$
are dense, $\gamma$ acts trivially on $\ol{Y}$, so $\gamma \in \Gamma^{st}$.

Now choose $\gamma \in \Gamma^{st}$.  By \cite[Remark 2.21]{Li:sr}, the
extension $K^{st}/K$ is the compositum of two
extensions: the minimal extension $K'/K$ leading to the stable reduction
$\ol{f}': \ol{Y}' \to \ol{X}'$
of $f$ under the definition of \cite{Ra:sp} (where we allow the ramification
points to coalesce), 
as well as the minimal extension $K''/K$ over which all of
the ramification points of $f$ are defined.  Since the ramification points of
$f$ specialize to distinct points on
$\ol{Y}$, it follows that $\gamma$ does not permute these points nontrivially. 
But any nontrivial element of 
$G(K''/K)$ does permute the ramification points nontrivially.  
Thus $\gamma \in G_{K''}$.  On the other hand, since $\gamma$ acts trivially on
$\ol{Y}$ (which dominates $\ol{Y}'$), 
it acts trivially on $\ol{Y}'$.  By \cite[Proposition 2.2.2]{Ra:sp}, $\gamma \in
G_{K'}$.  
Since $G_{K'} \cap G_{K''} = G_{K^{st}}$, we have $\gamma \in G_{K^{st}}$.
\end{proof}  
Since $\Gamma^{st}$ is the kernel of the homomorphism $G_k \to \Aut(\ol{Y})$, it
follows from Proposition 
\ref{Pstablecutout} that $K^{st}$ is Galois over $K$.

If $\ol{Y}$ is smooth, the cover $f: Y \to X$ is said to have \emph{potentially
good reduction}.  If $\ol{Y}$ can be contracted to a smooth curve by blowing
down components of
genus zero, then the curve $Y$ is said to have potentially good reduction.  If
$f$ or $Y$ does not have
potentially good reduction, it is said to have \emph{bad reduction}.  In any
case, the special fiber $\ol{f}:
\ol{Y} \to \ol{X}$ of the stable model is called the \emph{stable reduction} of
$f$.  
The action of $G$ on $Y$ extends to the stable reduction $\ol{Y}$ and $\ol{Y}/G
\cong \ol{X}$.  The strict transform of the special fiber of $X_{R^{st}}$ in
$\ol{X}$ is called the \emph{original
component}, and will be denoted $\ol{X}_0$.

\subsubsection{The graph of the stable reduction}
As in \cite{We:br}, we construct the (unordered) dual graph $\mc{G}$ of the
stable reduction of $\ol{X}$. 
An \emph{unordered graph} $\mc{G}$ consists of a set of \emph{vertices}
$V(\mc{G})$ and a set of \emph{edges} $E(\mc{G})$.  
Each edge has a \emph{source vertex} $s(e)$ and a \emph{target vertex} $t(e)$. 
Each edge has an \emph{opposite
edge} $\ol{e}$, such that $s(e) = t(\ol{e})$ and $t(e) = s(\ol{e})$.  Also,
$\ol{\ol{e}} = e$.

Given $f$, $\ol{f}$, $\ol{Y}$, and $\ol{X}$ as in this section, we construct two
unordered graphs $\mc{G}$ and
$\mc{G}'$.  In our construction, $\mc{G}$ has a vertex $v$ for each irreducible
component of $\ol{X}$ and an edge $e$ for each ordered triple $(\ol{x}, \ol{W}',
\ol{W}'')$, 
where $\ol{W}'$ and $\ol{W}''$ are irreducible components of $\ol{X}$ whose
intersection is $\ol{x}$.  If $e$
corresponds to $(\ol{x}, \ol{W}', \ol{W}'')$, then
$s(e)$ is the vertex corresponding to $\ol{W}'$ and $t(e)$ is the vertex
corresponding to
$\ol{W}''$.  The opposite edge of $e$ corresponds to $(\ol{x}, \ol{W}'',
\ol{W}')$.
We denote by $\mc{G}'$ the \emph{augmented} graph of $\mc{G}$ constructed as
follows: consider the set
$B_{\text{wild}}$ of branch points of $f$ with branching index divisible by $p$.
 
For each $x \in B_{\text{wild}}$, we know that $x$ specializes to 
a unique irreducible component $\ol{W}_x$ of $\ol{X}$, corresponding to a vertex
$A_x$ of $\mc{G}$.  
Then $V(\mc{G}')$ consists of the elements of $V(\mc{G})$ 
with an additional vertex $V_x$ for each $x \in B_{\text{wild}}$.  Also,
$E(\mc{G}')$ consists of the elements of 
$E(\mc{G})$ with two additional opposite edges for each $x \in
B_{\text{wild}}$, 
one with source $V_x$ and target $A_x$, and one with source $A_x$ and
target $V_x$.  We write $v_0$ for the vertex corresponding to the original
component $\ol{X}_0$.  

An irreducible component of $\ol{X}$ corresponding to a leaf of $\mc{G}$ that is
not $\ol{X}_0$ is 
called a \emph{tail} of $\ol{X}$.  All other components are called
\emph{interior components}.  We partially order the
vertices of $\mc{G}'$ such that $v_1 \preceq v_2$ iff $v_1 = v_2$, $v_1 = v_0$,
or $v_0$ and $v_2$ are in different connected components of $\mc{G}' \backslash v_1$ (we order ``outward" from the original component).  
Similarly, we can compare edges with each other, and edges with vertices.  For this we overload the symbol $\preceq$.
The set of irreducible components and singular points of $\ol{X}$ inherits the partial order $\preceq$.  

\subsubsection{Inertia Groups of the Stable Reduction.}

\begin{prop}[(\cite{Ra:sp}, Proposition 2.4.11)]\label{Pspecialram}
The inertia groups of $\ol{f}: \ol{Y} \to \ol{X}$ at points of $\ol{Y}$ are as
follows (note that points in the same
$G$-orbit have conjugate inertia groups):
\begin{description}
\item{(i)} At the generic points of irreducible components, the inertia groups
are $p$-groups.  
\item{(ii)} At each node, the inertia group is an extension of a cyclic,
prime-to-$p$ order group, by a $p$-group generated by
the inertia groups of the generic points of the crossing components.
\item{(iii)} If a point $y \in Y$ above a branch point $x \in X$ specializes
to a smooth point $\ol{y}$ on a 
component $\ol{V}$ of 
$\ol{Y}$, then the inertia group at $\ol{y}$ is an extension of the
prime-to-$p$ part of the inertia group at $y$ by
the inertia group of the generic point of $\ol{V}$.
\item{(iv)} At all other points $q$ (automatically smooth, closed), the inertia
group is equal to the inertia group of the 
generic point of the irreducible component of $\ol{Y}$ containing $q$.
\end{description}
\end{prop}
If $\ol{V}$ is an irreducible component of $\ol{Y}$, we will always write
$I_{\ol{V}} \leq G$ for the inertia group of
the generic point of $\ol{V}$, and $D_{\ol{V}}$ for the decomposition group.

For the rest of this subsection, assume $G$ has a \emph{cyclic} $p$-Sylow
subgroup.
When $G$ has a cyclic $p$-Sylow subgroup, the inertia groups above a generic
point of an
irreducible component $\ol{W} \subset 
\ol{X}$ are conjugate cyclic groups of $p$-power order.  If they are of order
$p^i$, we call $\ol{W}$ a
\emph{$p^i$-component}.  If $i = 0$, we call $\ol{W}$ an \emph{\'{e}tale
component}, and if $i > 0$, we call $\ol{W}$ an
\emph{inseparable component}. 	 
For an inseparable component $\ol{W}$, the morphism $Y \times_X \ol{W} \to
\ol{W}$ corresponds to an inseparable extension 
of the function field $k(\ol{W})$.
This is because, since $\ol{Y}$ is reduced, the inertia of $f$ at the local ring
of the generic point of an
irreducible component of $\ol{Y}$ above $\ol{W}$ must come from an inseparable
extension of residue fields.

An \'{e}tale tail of $\ol{X}$ is called \emph{primitive} if it contains a branch
point other
than the point at which it intersects the rest of $\ol{X}$.  
Otherwise it is called \emph{new}.  This follows \cite{We:br}.  An inseparable
tail that does not contain the
specialization of any branch point will be called a \emph{new inseparable tail}.
 
A inseparable tail that is a $p^i$-component will also be called a
\emph{$p^i$-tail} (a \emph{new $p^i$-tail} if it is new).

\begin{corollary} \label{Cconsistency}
If $\ol{V}$ and $\ol{V}'$ are two adjacent irreducible components of $\ol{Y}$,
then either $I_{\ol{V}} \subseteq I_{\ol{V}'}$ or vice versa.
\end{corollary}
\begin{proof} Let $q$ be a point of intersection of $\ol{V}$ and $\ol{V}'$ and
let $I_q$ be its inertia group.  
Then the $p$-part of $I_q$ is a cyclic $p$-group, generated by the two
cyclic $p$-groups $I_{\ol{V}}$ and $I_{\ol{V}'}$.  
Since the subgroups of a cyclic $p$-group 
are totally ordered, the corollary follows. \end{proof}

\begin{corollary}\label{Cinsepnormalp}
Let $\ol{S} \subseteq \ol{Y}$ be a union of irreducible components of $\ol{Y}$,
all of which lie above inseparable
components of $\ol{X}$.  Suppose that $\ol{S}$ is connected.  Let $D_{\ol{S}}
\subseteq G$ be the decomposition group of
$\ol{S}$ (i.e., the maximal subgroup of $G$ such that $D_{\ol{S}}(\ol{S}) =
\ol{S}$).  Then $D_{\ol{S}}$ has a normal 
subgroup of order $p$.
\end{corollary}

\begin{proof} 
Pick any irreducible component $\ol{V}$ of $\ol{S}$, and let $Q$ be the unique
subgroup of order $p$ of $I_{\ol{V}}$.
Let $g \in D_{\ol{S}}$, and write $\ol{V}' = g\ol{V}$.  
Then $gQg^{-1}$ is the unique subgroup of order $p$ of the inertia group of the
generic point of $\ol{V}'$.  
Since $\ol{S}$ is connected, there exists a sequence $\Sigma$ of components of
$\ol{S}$, 
starting with $\ol{V}$ and ending with $\ol{V}'$, such that each
component in $\Sigma$ intersects the preceeding and the following component. 
The components in $\Sigma$ all lie
above inseparable components of $\ol{X}$.  
Then the inertia group of the generic point of each component in $\Sigma$ has a
unique subgroup of order $p$.
We know from Corollary \ref{Cconsistency} that for any two such adjacent
components, the inertia group of one contains the inertia group of the other. 
Thus, both inertia groups contain the 
\emph{same} subgroup of order $p$.  So $Q = gQg^{-1}$, and we are done.  
\end{proof}

\begin{prop} \label{Pcorrectspec}
If $x \in X$ is branched of index $p^as$, where $p \nmid s$, then $x$
specializes to a $p^a$-component.
\end{prop}

\begin{proof}
By Proposition \ref{Pspecialram} (iii) and the definition of the stable model, 
$x$ specializes to a smooth point of a component whose generic inertia has order
at least $p^a$.  
Because our definition of the stable model requires the specializations of the
$|G|/p^{a}s$ 
ramification points above $a$ to be disjoint, the specialization of $x$ must
have a fiber with cardinality a multiple of
$|G|/p^a s$.  This shows that $x$ must specialize to a component with inertia at
most $p^a$.
\end{proof}

\begin{remark}\label{Rcorrectspec}
It follows from Proposition \ref{Pspecialram} (iii) and the proof of Proposition
\ref{Pcorrectspec} that if $y$ is a 
ramification point above $x$, then the specialization $\ol{y}$ of $y$ also has
inertia group in $G$ cyclic of order
$p^a s$.
\end{remark}

\begin{lemma}[(\cite{Ra:sp}, Proposition 2.4.8)]\label{Letaletail}
If $\ol{W}$ is an \'{e}tale component of $\ol{X}$, then $\ol{W}$ is a tail.
\end{lemma}
\smallskip
\begin{lemma} \label{Ltailetale}
If $\ol{W}$ is a $p^a$-tail of $\ol{X}$, then the component $\ol{W}'$ that
intersects $\ol{W}$ is a $p^b$-component with
$b > a$.
\end{lemma}
\begin{proof}   
The proof is essentially the same as the proof of \cite[Lemme 3.1.2]{Ra:sp}.  
Assume that the proposition is false.  
Let $\ol{V}$ be an irreducible component of $\ol{Y}$ lying above the genus zero
component $\ol{W}$.  By Proposition \ref{Pspecialram} and our assumption, the
map $g:\ol{V} \to \ol{W}$ is the
composition $h \circ q$ of a tamely ramified, generically \'{e}tale morphism $h$
with a radicial morphism $q$ of degree $p^a$  
Now, $h$ can only be branched at the intersection $\ol{w}$ of $\ol{W}$ and
$\ol{W}'$, 
and the specialization of, at most, one point $a_i$ to $\ol{W}$.  Since there
are
at most two branch points, and they are tame, then $h$
is totally ramified at these points.  So $q(\ol{V})$ has genus zero and has only
one point above $\ol{w}$.
Since $q$ is radicial, the same is true for $\ol{V}$.  This
contradicts the definition of the stable model, as $\ol{V}$ has genus zero and
insufficiently many marked points. \end{proof}

Note that Lemma \ref{Ltailetale} shows that if $p$ exactly divides $|G|$,
then there are no inseparable tails.
But there can be inseparable tails if a higher power of $p$ divides $|G|$.

\begin{notation}\label{Nsx}
Let $x$ be the intersection point of two components $\ol{W}$ and $\ol{W}'$ of
$\ol{X}$, and let $y$ lie above $x$, on
the intersection of two components $\ol{V}$ and $\ol{V}'$ of $\ol{Y}$.
Assume that $\ol{W}$ is a $p^r$-component and $\ol{W}'$ is a $p^{r'}$-component,
$r \geq r'$.
The inclusion $\hat{\mc{O}}_{\ol{W}', x} \hookrightarrow \hat{\mc{O}}_{\ol{V}',
y}$ induced from the cover is a 
composition $$\hat{\mc{O}}_{\ol{W}', x} \hookrightarrow S \hookrightarrow
\hat{\mc{O}}_{\ol{V}', y}$$
where $\hat{\mc{O}}_{\ol{W}', x} \hookrightarrow S$ is a totally ramified Galois
extension with group 
$J \cong \ints/p^{r-r'} \rtimes \ints/m_x$ and $S \hookrightarrow
\hat{\mc{O}}_{\ol{V}',
y}$ is a purely inseparable extension of degree $p^{r'}$.  
The extension $\hat{\mc{O}}_{\ol{W}', x} \hookrightarrow S$ and the group $J$
depend only on $x$, up to isomorphism, so
we denote them by $\hat{\mc{O}}_{\ol{W}', x} \hookrightarrow S_x$, and $J_x$,
respectively.
\end{notation}

\begin{definition}\label{Draminvariant}
For $r > r'$, let $B_{r,r'}$ index the set of intersection points of a
$p^r$-component $\ol{W}$ with a $p^{r'}$-component $\ol{W}'$ of $\ol{X}$.  
For $b \in B_{r,r'}$, let $x_b$ be the corresponding point of intersection. 
For any $r' \leq \alpha \leq r$, let $I_{\alpha} \leq J_x$ be the unique
subgroup of $J_x$ of order $p^{\alpha - r'}$.  Finally, for $r' \leq \alpha <
r$, let 
$\sigma^{\alpha}_b$ be the conductor of the extension $\mc{O}_{\ol{W}', x_b}
\hookrightarrow 
S^{I_{\alpha}}_{x_b}$ (\S\ref{Sramification}).  
If $\alpha = r'$, we will often just write $\sigma_{x_b}$ for
$\sigma^{r'}_{x_b}$. 
Furthermore, if $x_b$ lies on a tail $\ol{X}_b$, we will simply write $\sigma_b$
(resp.\ $\sigma^{\alpha}_b$) 
for $\sigma_{x_b}$ (resp.\ $\sigma^{\alpha}_{x_b}$).  

We call the $\sigma^{\alpha}_{x_b}$ the \emph{truncated effective ramification
invariants} (of $\ol{f}$) at 
$x_b$.  We call $\sigma_{x_b}$ simply the \emph{effective ramification
invariant} (of $\ol{f}$) at $x_b$.  If
$\ol{X}_b$ is a tail of $\ol{X}$, then $\sigma^{\alpha}_b$ (resp.\ $\sigma_b$)
is called the truncated effective ramification
invariant (resp.\ effective ramification invariant) of the tail $\ol{X}_b$.
\end{definition}

\begin{remark}
In the case $r = 1, r' = 0$, the $\sigma_b$ for tails $\ol{X}_b$ are the same as
those defined in \cite{Ra:sp} and
\cite{We:br}.
\end{remark}

\begin{lemma}\label{Lramdenominator}
The effective ramification invariants $\sigma^{\alpha}_{x_b}$ lie in
$\frac{1}{m_{J_{x_b}}}\ints$.  In particular, they
lie in $\frac{1}{m_G} \ints$.
\end{lemma}

\begin{proof}
The extension $S_{x_b}^{I_{\alpha}}/\hat{\mc{O}}_{\ol{X}_b, x_b}$ has Galois
group $J_{x_b}/I_{\alpha}$,
which is isomorphic to $\ints/p^d \rtimes \ints/\ell$ for some $d, \ell$.  The
quotient of $J_{x_b}/I_{\alpha}$
by its maximal prime-to-$p$ central subgroup $H$ is $\ints/p^d \rtimes
\ints/m_{J_{x_b}}$.
The effective ramification invariants over $x_b$ are not affected by quotienting
out by $H$, as the upper numbering is 
invariant under taking quotients.  So $\sigma^{\alpha}_{x_b} \in
\frac{1}{m_{J_{x_b}}}$ by Lemma \ref{Lhassearf}.
Since $J_{x_b}$ is a subquotient of $G$, it is easy to see that $m_{J_{x_b}} |
m_G$, showing the second statement of the
lemma.
\end{proof}

We give one more definition:
\begin{definition}\label{Dmonotonic}
Let $\ol{W}$ be an irreducible component of $\ol{X}$.  We call the stable
reduction $\ol{f}$ of $f$ \emph{monotonic from
$\ol{W}$} if for every $\ol{W} \preceq \ol{W}' \preceq \ol{W}''$ such that
$\ol{W}'$ is a $p^i$-component and $\ol{W}''$
is a $p^j$-component, we have $i \geq j$.  In other words, the stable reduction
is monotonic from $\ol{W}$ if the
generic inertia does not increase as we move outward from $\ol{W}$ along
$\ol{X}$.  If $\ol{f}$ is monotonic from the
original component $\ol{X}_0$, we say simply that $\ol{f}$ is \emph{monotonic}.
\end{definition}

\begin{question}
Is the stable reduction of $f$ always monotonic?
\end{question}

This is true, for instance, when $G$ is $p$-solvable 
(i.e., has no non-abelian composition factors of order divisible by $p$) (\cite[Proposition
2.8]{Ob:fm1}).

\section{Deformation data and vanishing cycles formulas}\label{Sreduction}
Let $R$, $K$, and $k$ be as in \S\ref{Sstable},  and let $\pi$ be a uniformizer
of
$R$.  
Assume further that $R$ contains the $p$th roots of unity.
For any scheme or algebra $S$ over $R$, write $S_K$ and $S_k$ for its base
changes to $K$ and $k$, respectively.
Recall that we normalize the valuation of $p$ (not $\pi$) to be 1.  Then
$v(\pi) = 1/e$, where $e$ is the absolute ramification index of $R$.  

\subsection{Reduction of $\mu_p$-torsors}

The following result, Proposition \ref{Pmupreduction}, is used only in
Construction \ref{CCdefdata} (once).  
While essential, the proposition may be skipped on a first reading and its input
into Construction \ref{CCdefdata} 
accepted as a ``black box."

\begin{prop}[(\cite{He:ht}, Chapter 5, Proposition 1.6)]\label{Pmupreduction}
Let $X = \Spec A$ be a flat affine scheme over $R$, with relative dimension
$\leq 1$ and
integral fibers.
We suppose that $A$ is a factorial $R$-algebra that is complete with respect to
the $\pi$-adic valuation.
Let $Y_K \to X_K$ be a non-trivial \'{e}tale $\mu_p$-torsor, given by an
equation
$y^p = f$, where $f$ is invertible in $A_K$.
Let $Y$ be the normalization of $X$ in $Y_K$; we suppose the special fiber of
$Y$ is integral (in particular, reduced).
Let $\eta$ (resp.\ $\eta'$) be the generic point of the special fiber of $X$
(resp.\ $Y$).  
The local rings $\mc{O}_{X, \eta}$ and $\mc{O}_{Y, \eta'}$ are thus discrete
valuation rings with uniformizer $\pi$.  
Write $\delta$ for the valuation of the different of $\mc{O}_{Y,
\eta'}/\mc{O}_{X, \eta}$.  
We then have two cases, depending on the value of $\delta$ (which always
satisfies $0 \leq \delta \leq 1$):

\begin{itemize}
\item If $\delta = 1$, then $Y \cong \Spec B$, with $B = A[y]/(y^p - u)$, for
$u$ a
unit in $A$, not congruent to a $p$th power modulo $\pi$, and unique up to
multiplication
by a $p$th power in $A^{\times}$.  We say that the torsor $Y_K \to X_K$ has
\emph{multiplicative reduction}.
\item If $0 \leq \delta < 1$, then $\delta = 1 - n(\frac{p-1}{e})$, where $n$ is
an integer such that $0 < n \leq
e/(p-1)$.  Then $Y \cong \Spec B$, with $$B = \frac{A[w]}{(\frac{(\pi^n w+1)^p
-1}{\pi^{pn}} - u)},$$ for $u$ a unit of
$A$, not congruent to a $p$th power modulo $\pi$.  Also, $u$ is unique in the
following sense:  If an element $u' \in
A^{\times}$ could take the place of $u$, then
there exists $v \in A$ such that $$\pi^{pn}u' + 1 = (\pi^{pn}u + 1)(\pi^nv +
1)^p.$$  
If $\delta > 0$ (resp.\ $\delta = 0$), we say that the torsor $Y_K \to X_K$ has
\emph{additive reduction} (resp. \emph{\'{e}tale reduction}).
\end{itemize}
\end{prop}

\begin{remark}\label{Rvaluation} 
\begin{description}
\item{(i)}
In \cite{He:ht}, Proposition \ref{Pmupreduction} is stated for $X \to \Spec R$
with dimension 1
fibers, but the proof carries over without
change to the case of dimension $0$ fibers as well (i.e., the case where $A$ is
a discrete valuation ring containing
$R$).  It is this case that will be used in \S\ref{Sdefdata} to define
deformation data.
\item{(ii)} In the cases of multiplicative and additive reduction, the
map $Y_k \to X_k$ is seen to be inseparable. 
\end{description}
\end{remark}

\subsection{Deformation Data}\label{Sdefdata}
Deformation data arise naturally from the stable reduction of covers.  Say $f:Y
\to X$ is a branched $G$-cover as in
\S\ref{Sstable}, with stable model $f^{st}: Y^{st} \to X^{st}$ and stable
reduction $f: \ol{Y} \to \ol{X}$.  Much information is lost when we pass from
the stable model to the stable
reduction, and deformation data provide a way to retain some of this
information.  This process is described in detail in
\cite[5, \S1]{He:ht} in the case where the inertia group of a component has
order $p$.  In Construction \ref{CCdefdata},
we generalize it to the case where the inertia group is cyclic of order $p^r$.

\subsubsection{Generalities}
Let $\ol{W}$ be any connected smooth proper curve over $k$.  
Let $H$ be a finite group and $\chi$ a 1-dimensional character 
$H \to \FF_p^{\times}.$  A \emph{deformation datum} over
$\ol{W}$ of type $(H, \chi)$ is an ordered pair $(\ol{V}, \omega)$ such that:
$\ol{V} \to \ol{W}$ is
an $H$-Galois branched cover;
$\omega$ is a meromorphic differential form on $\ol{V}$ that is either
logarithmic or
exact (i.e., $\omega = du/u$ or $du$ for 
$u \in k(\ol{V})$); and $\eta^*\omega = \chi(\eta)\omega$ for all $\eta \in H$. 
If $\omega$
is logarithmic (resp.\ exact), the deformation datum is called
multiplicative (resp.\ additive).  When $\ol{V}$ is understood, we will
sometimes
speak of the deformation datum $\omega$.  

If $(\ol{V}, \omega)$ is a deformation datum, and $w \in \ol{W}$ is a closed
point, we
define $m_w$ to be the order of the 
prime-to-$p$ part of the ramification index of $\ol{V} \to \ol{W}$ at $w$. 
Define $h_w$
to be $\ord_v(\omega) + 1$, where $v \in
\ol{V}$ is any point which maps to $w \in \ol{W}$.  This is well-defined because
$\omega$
transforms nicely via $H$.  Lastly, define $\sigma_x = h_w/m_w$.  We call $w$ a
\emph{critical point} of the
deformation datum $(\ol{V}, \omega)$ if 
$(h_w, m_w) \ne (1, 1)$.  Note that every deformation datum contains only a
finite number of critical points.  The
ordered pair $(h_w, m_w)$ is called the \emph{signature} of $(\ol{V}, \omega)$
(or of
$\omega$, if $\ol{V}$ is understood) at $w$, and
$\sigma_w$ is called the \emph{invariant} of the deformation datum at $w$.

\begin{prop}\label{Pdefdatadenom}
Let $(\ol{V}, \omega)$ be a deformation datum of type $(H, \chi)$.  Let $v \in
\ol{V}$ 
be a tamely ramified point lying over $w \in \ol{W}$, and write $I_v$ for the
inertia group of $\phi: \ol{V} \to \ol{W}$ at 
$v$.  If $|I_v/(I_v \cap \ker(\chi))| = \mu$, then $\sigma_w \in
\frac{1}{\mu}\ints$.
\end{prop}

\begin{proof}
In a formal neighborhood of $v$, we can use Kummer theory to see that $\phi$ is
given by the equation 
$k[[t]] \hookrightarrow k[[t]][\tau]/(\tau^{m_w} - t)$, where $t$ is a local
parameter at $w$ and $\tau$ 
is a local parameter at $v$.  Expanding $\omega$ out as a Laurent series in
$\tau$, we can write 
$$\omega = \left(c\tau^{h_w-1} + \sum_{i=1}^{\infty} c_i \tau^{h_w - 1 +
i}\right) d\tau.$$
Let $g$ be a generator of $I_v$ such that $g^*(\tau) = \zeta_{m_w} \tau$.  Since
$g^{\mu} \in \ker(\chi)$, we have that $(g^{\mu})^*\omega = \omega$.
Thus $(g^{\mu})^* (\tau^{h_w-1} d\tau) = \tau^{h_w - 1} d\tau$.  So $\mu h_w$ is
a multiple of $m_w$.  Therefore,
$\sigma_w = \frac{h_w}{m_w} \in \frac{1}{\mu}\ints$.
\end{proof}

\subsubsection{Deformation data arising from stable reduction.}
Maintain the notations of \S\ref{Sstable}.  For each irreducible component of
$\ol{Y}$ lying above a
$p^r$-component of $\ol{X}$ with $r > 0$, we will construct $r$ different
deformation data. 
For this construction, we can replace $K^{st}$ with a finite extension $K'$ that
is as large as we wish.  
In particular, we work over $K'$ containing a $p$th root of unity and having
ring of integers $R'$.  By abuse of notation,
we write $Y^{st}$ for $Y^{st} \times_{K^{st}} K'$.

\begin{construction}\label{CCdefdata}
Let $\ol{V}$ be an irreducible component of $\ol{Y}$ with generic point $\eta$
and nontrivial generic inertia group 
$I \cong \ints/p^r \subset G$.
Write $B = \hat{\mc{O}}_{Y^{st}, \eta}$, and $C = B^I$, the invariants of $B$
under
the action of $I$.  Then $B$ (resp.\ $C$) is a complete, mixed characteristic, 
discrete valuation ring with residue field $k(\ol{V})$ (resp.\
$k(\ol{V})^{p^r}$).
The group $I \cong \ints/p^r$ acts on $B$; 
for $0 \leq i \leq r$, we write $I_i$ for the subgroup of order $p^i$ in $I$,
and we
write $C_i$ for the fixed ring $B^{I_{r-i+1}}$.  
Thus $C_1 = C$.  Then for $1 \leq i \leq r$, the extension  $C_{i}
\hookrightarrow C_{i+1}$ is an
extension of complete discrete valuation
rings satisfying the conditions of Proposition \ref{Pmupreduction}, but with
relative dimension 0 instead of
1 over $R'$ (see Remark \ref{Rvaluation} (i)).  On the generic fiber, the
extension is
given by an equation $y^p = z$, where 
$z$ is well-defined up to raising to a prime-to-$p$ power in
$C_{i}^{\times}/(C_{i}^{\times})^p$.  We make $z$ completely
well-defined in $C_{i}^{\times}/(C_{i}^{\times})^p$ by fixing a $p$th root of
unity
$\mu$ and a generator $\alpha$ of
Aut($C_{i+1}/C_{i}$) and forcing $\alpha(y) = \mu y$.  In both the case of
multiplicative and additive reduction, Proposition \ref{Pmupreduction} yields an
element $$\ol{u} \in C_{i}
\otimes_{R'} k = k(\ol{V})^{p^{r-i+1}} \cong k(\ol{V})^{p^r},$$ the last
isomorphism coming from raising to the $p^{i-1}$st
power.  In the case of multiplicative reduction, set $\omega_i =
d\ol{u}/\ol{u}$, and in the case of additive reduction, 
set $\omega_i = d\ol{u}$.  In both cases, $\omega_i$ can be viewed as a
differential form on $k(\ol{V})^{p^r}$. 
Write $\ol{V}'$ for the curve whose function field 
is $C \otimes_{R'} k = k(\ol{V})^{p^r} \subset k(\ol{V})$.  Then each
$\omega_i$ is a meromorphic differential form on 
$\ol{V}'$.

Furthermore, let $D := D_{\ol{V}}$, and write $H = D/I$.  
Then if $\ol{W}$ is the component of $\ol{X}$ lying below $\ol{V}$, we have maps
$\ol{V} \to \ol{V}' \to \ol{W}$, with
$\ol{W} = \ol{V}'/H$.  The curves $\ol{V}$ and $\ol{V}'$ are abstractly
isomorphic.  
Any $g \in H$ has a canonical conjugation action on
$I$, and also on the subquotient of $I$ given by $\Aut(C_{i+1}/C_i)$.  This
action gives a
homomorphism $\chi: H \to \FF_p^{\times}$.  We claim to have constructed, for
each $i$, a deformation datum $(\ol{V}', \omega_i)$ of type $(H, \chi)$ over
$\ol{W}$.

Everything is clear except for the transformation property, so let $g \in H$. 
Then for $z$ as in the construction, taking a $p$th root of $z$ and of $g^*z$
must yield the same
extension, so $g^*z = c^pz^q$ with 
$c \in C_i$ and $q \in \{1, \ldots, p-1\}$.  It follows that $g^*y = \zeta cy^q$
for $\zeta$ some $p$th root of unity.  It also follows that $g^*(\omega_i) =
q\omega_i$.
If $\alpha$ is a generator of $\Aut(C_{i+1}/C_i)$ as before, then we must show
that $g \alpha g^{-1} = \alpha^q$.

Write $\alpha^*y = \mu y$ for some, possibly different, $p$th root of unity
$\mu$.  Then  
$$(g\alpha g^{-1})^*(y) = (g^{-1})^* \alpha^* g^*y = (g^{-1})^* \alpha^* \zeta
cy^q = (g^{-1})^* \mu^q\zeta c y^q = \mu^q
y.$$  Thus $g\alpha g^{-1} = \alpha^q$.  This completes Construction
\ref{CCdefdata}.
\end{construction}
\vspace{.35in}

For the rest of this section, we will only concern ourselves with deformation
data that arise from the stable reduction of 
branched $G$-covers $Y \to X = \proj^1$ where $G$ has a cyclic $p$-Sylow
subgroup,
via Construction \ref{CCdefdata}.
We will use the notations of \S\ref{Sstable} and Construction \ref{CCdefdata}
throughout this section.

From \cite[Proposition 1.7]{We:br}, we have the following result in the case of
inertia groups of order $p$.  The proof is the same in our case, and we omit it.
\begin{lemma}\label{Lcritical}
Say $(\ol{V}', \omega)$ is a deformation datum arising from the stable reduction
of a cover as in Construction
\ref{CCdefdata}, and let $\ol{W}$ be the component of $\ol{X}$ lying under
$\ol{V}'$.  Then a critical point $x$ of the
deformation datum on $\ol{W}$ is either a singular point of $\ol{X}$ or the
specialization of a branch point of $Y \to
X$ with ramification index divisible by $p$.  In the first case, $\sigma_x \ne
0$, and in the second case, $\sigma_x = 0$
and $\omega$ is logarithmic.
\end{lemma}

The next result, Proposition \ref{Pannuluscompatibility},
generalizes one part of the theorem \cite[5, Theorem 1.10]{He:ht}.  It provides
the inner workings behind the cleaner 
interface given by Lemma \ref{Lsigmaeffcompatibility}.  
We assume the situation of Notation \ref{Nsx} (in particular,
the notations $x$, $y$, $\ol{W}$, $\ol{W}'$, $\ol{V}$, $\ol{V}'$, $r$, $r'$,
$S_x$, and $J_x$). 
By Lemma \ref{Letaletail}, $r \geq 1$.
For each $i$, $1 \leq i \leq r$, there is a deformation datum with
differential form $\omega_i$ associated
to $\ol{V}$.  For each $i'$, $1 \leq i' \leq r'$, there is a deformation datum
$\omega'_{i'}$
associated to $\ol{V}'$.  Let $m_x$ be the prime-to-$p$ part of the ramification
index at $x$.  
Let $I$ be the inertia group of $y$ in $G$, and let $I_i$ (resp.\ $J_i$) 
be the unique subgroup of order $p^i$ in $I$ (resp.\ $J_x$).
The following proposition gives a compatibility between deformation data, and
also relates
deformation data to the geometry of $\ol{Y}$.

\begin{prop}\label{Pannuluscompatibility}
With $x$ as above, let $(h_{i,x}, m_x)$ (resp.\ $(h'_{i', x}, m_x)$) be the
signature of $\omega_i$
(resp.\ $\omega'_{i'}$) at $x$.
Write $\sigma_{i,x} = h_{i,x}/m_x$ and $\sigma'_{i',x} = h'_{i',x}/m_x$.
Then the following hold:
\begin{description}
\item{(i)} If $i = i' + r - r'$, then $h_{i,x} = -h'_{i',x}$ and $\sigma_{i,x} =
-\sigma'_{i',x}$.
\item{(ii)} If $i \leq r - r'$, then $h_{i,x} = h$, where $h$ is the upper
(equivalently lower) jump in the extension
$S_x^{J_{r-r'-i+1}} \hookrightarrow S_x^{J_{r-r'-i}}$.
Also, $\sigma_{i,x} = \sigma$, where $\sigma$ is the upper jump in the extension
$S_x^{J_{r - r' - i + 1} \rtimes \ints/m_x} \hookrightarrow S_x^{J_{r-r'-i}}$.
\end{description}
\end{prop}

\begin{proof}
(cf. \cite[Proposition 1.8]{We:br})
The group $I$ acts on the annulus $\mc{A} = \Spec \hat{\mc{O}}_{Y^{st}, y}$.
The statements about $h_{i,x}$ follow from \cite[5, Proposition 1.10]{He:ht}
applied to
the formal annulus $\mc{A}/(I_{r -i})$ and an automorphism given by a generator
of
$I_{r-i+1}/I_{r-i}$ considered as a subquotient of $J_x$.  
The statements about $\sigma_{i,x}$ follow from dividing the statements about
$h_{i,x}$ by $m_x$.  
Note that what we call $h_{i,x}$, \cite{He:ht} calls $-m$.
\end{proof}

\begin{remark}\label{Rupperlower}
For $r' \leq \alpha < r$, consider the $\ints/p^{r-\alpha} \rtimes
\ints/m_x$-extension 
$\hat{\mc{O}}_{\ol{W}', x} \hookrightarrow S_x^{J_{\alpha}}$.
If the $j_i$ are its lower jumps (see \S\ref{Sramification}), then Proposition
\ref{Pannuluscompatibility}, combined
with \cite[Lemma 3.1]{OP:wc}, shows that $j_i = h_{i, x}$.  By Remark
\ref{Rlower2conductor}, the conductor of this
extension is equal to
$$\left(\sum_{i=1}^{r-\alpha-1} \frac{p-1}{p^{i}}\sigma_{i,x} \right) +
\frac{1}{p^{r-\alpha-1}}\sigma_{r-\alpha, x}.$$
\end{remark}
\vspace{.35in}

We set up the \emph{local vanishing cycles formula}.  Our first version,
Equation (\ref{Elocvancycles}), will be
unwieldy, but it will be used to prove our second version, the much cleaner
Equation (\ref{Egenlocvancycles}).
Let $(\ol{V}, \omega)$ be a deformation datum of type $(H, \chi)$
(not necessarily coming from the stable reduction of a cover). 
Let $B$ (resp.\ $B'$) be the set of critical points of $(\ol{V}, \omega)$ where
$\ol{V} \to \ol{W}$ is tamely (resp.
wildly) ramified.  Let $g_W$ be the genus of $\ol{W}$.  

For each $w \in B'$, suppose that the inertia group of a point $v$ above $w$ is 
$\ints/p^{n_w} \rtimes \ints/m_w$ with $p \nmid m_w$.  For $1 \leq i \leq n_w$,
let $h_{i, w}$ be the $i$th lower jump of 
the extension $\hat{\mc{O}}_{\ol{V}, v}/\hat{\mc{O}}_{\ol{W}, w}$, and let
$\sigma_{i,w} = h_{i, w}/m_w$.  
We maintain the notation $(h_w, m_w)$ for the signature of $\omega$ at $w$, and
$\sigma_w$ for the invariant at $w$ (note that there is not necessarily any
relation between the $\sigma_{i, w}$ and $\sigma_w$). Then we have

\begin{lemma}[(Local vanishing cycles formula)]\label{Llocvancycles}
\hspace*{\fill}
\begin{equation}\label{Elocvancycles}
\sum_{w \in B'} \left(\frac{\sigma_w}{p^{n_w}} - 1 - \sum_{i=1}^{n_w}
\frac{p-1}{p^i}\sigma_{i,w} \right)
+ \sum_{b \in B} (\sigma_b - 1) = 2g_W - 2.
\end{equation}
\end{lemma}

\begin{proof}
Let $g_V$ be the genus of $\ol{V}$, and $d$ the degree of the map $\ol{V} \to
\ol{W}$.  The
Hurwitz formula, along with Lemma \ref{Ljumpsdifferent} (i), yields that 
$$2g_V - 2  = d(2g_W - 2) + d\sum_{b \in B} (1 - \frac{1}{m_b})
+ \sum_{w \in B'} \frac{d}{p^{n_w}m_w} (p^{n_w}m_w-1 + \sum_{i=1}^{n_w}
h_{i,w}p^{{n_w}-i}(p-1)).$$
Furthermore, the degree of a differential form on $\ol{V}$ is $$\sum_{b \in
B} \frac{d}{m_b}(h_b - 1) + \sum_{w \in B'} \frac{d}{p^{n_w}m_w}(h_w - 1).$$  
Substituting this for $2g_V - 2$ and rearranging yields the formula.
\end{proof}

\begin{remark}
In the case $B' = \emptyset$, the local vanishing cycles formula
(\ref{Elocvancycles}) reduces to that found in
\cite[p.\ 998]{We:br}.
\end{remark}

Let us resume the assumption that all of our deformation data come from
Construction \ref{CCdefdata}.
Recall that $\mc{G}'$ is the augmented dual graph of $\ol{X}$.  To each edge $e$
of $\mc{G}'$ we will associate an
invariant $\sigma^{\eff}_e$, called the \emph{effective invariant}.  

For each $0 \leq j < n$, write $\mc{G}'_j$ \label{TEgraph} for 
the subgraph of $\mc{G}'$ consisting of: those vertices corresponding to
$p^s$-components for $s > j$; those
corresponding to specializations of branch points where $p^{j+1}$ divides the
branching index; and the edges incident to at
least one of these vertices.  Write $\mc{G}_j = \mc{G}'_j \cap \mc{G}$.  
Note that an edge in $E(\mc{G}'_j)$ might have a source or a target not in
$V(\mc{G}'_j)$;
these edges correspond to points of $B_{r, r'}$ with $r > j \geq r'$ (see
Definition \ref{Draminvariant}).  
Note further that $E(\mc{G}'_{0}) = E(\mc{G}')$. 
For each edge in $\mc{G}'_{j}$, we associate a set of invariants $\sigma^{\eff,
\alpha}_e$, $0 \leq \alpha < j$, called 
the \emph{truncated effective invariants}.  
The effective invariant $\sigma^{\eff}_e$ will be equal to the truncated
effective invariant $\sigma^{\eff, 0}_e$.

\begin{definition}\label{Dsigmaeff}
\begin{itemize}
\item If $s(e)$ corresponds to a $p^r$-component $\ol{W}$, and $t(e)$
corresponds to a $p^{r'}$-component $\ol{W}'$ with 
$r \geq r'$, then $r \geq 1$ by Lemma \ref{Letaletail}. 
Let $\omega_i$, $1 \leq i \leq r$, be the deformation data above $\ol{W}$.  
If $\{w\} = \ol{W} \cap \ol{W}'$, then 
$$\sigma^{\eff}_e := \left( \sum_{i=1}^{r-1} \frac{p-1}{p^{i}}\sigma_{i,w}
\right) + \frac{1}{p^{r-1}}\sigma_{r,w}.$$
Note that this is a weighted average of the $\sigma_{i,w}$'s.
Furthermore, we write
$$\sigma^{\eff, \alpha}_e := \left( \sum_{i=1}^{r-\alpha - 1}
\frac{p-1}{p^{i}}\sigma_{i,w} \right) +
\frac{1}{p^{r-\alpha -1}}\sigma_{r-\alpha,w}$$ for all $0 \leq \alpha < r$.
\item If $s(e)$ corresponds to a $p^r$-component and $t(e)$ corresponds to a
$p^{r'}$-component with $r < r'$, then
$\sigma^{\eff}_e := -\sigma^{\eff}_{\ol{e}}$.  Also, 
$\sigma^{\eff, \alpha}_e := -\sigma^{\eff, \alpha}_{\ol{e}}$ for all $\alpha <
r'$.
\item If either $s(e)$ or $t(e)$ is a vertex of $\mc{G}'$ but not $\mc{G}$, then
$\sigma^{\eff}_e := 0$.  
If, additionally, $e \in E(\mc{G}'_j)$, then $\sigma^{\eff, \alpha}_e := 0$ for
all $\alpha < j$.
\end{itemize}
\end{definition}

Essentially, the truncated effective invariants are the same as the regular
effective invariants, but we ignore the ``top"
$\alpha$ differential forms (and thus the truncated invariants are not defined
unless Construction \ref{CCdefdata}
associates more than $\alpha$ differential forms to the component in question).

\begin{lemma}\label{Lsigmaeffcompatibility}
\begin{description}
\item{(i)} For any $e \in E(\mc{G}')$, we have $\sigma^{\eff, \alpha}_e =
-\sigma^{\eff, \alpha}_{\ol{e}}$.
\item{(ii)} Supppose $e$ corresponds to a point $x$, $s(e)$ corresponds to a
$p^r$-component, $t(e)$ corresponds to a 
$p^{r'}$-component, and $r > r'$.  Then $\sigma^{\eff, \alpha}_e =
\sigma^{\alpha}_x$ for any $r' \leq \alpha < r$
(Definition \ref{Draminvariant}).
\item{(iii)}  In particular, if $t(e)$ corresponds to an \'{e}tale tail
$\ol{X}_b$, 
then $\sigma^{\eff}_e = \sigma_b$.
\end{description}
\end{lemma}

\begin{proof}
\emph{To (i):}  
This needs proof only when $e$ corresponds to the intersection of two
$p^r$-components.  But then the
result follows immediately from Proposition \ref{Pannuluscompatibility} (setting
$r = r'$).
\\
\\
\emph{To (ii):} 
In this case, $\sigma^{\eff, \alpha}_{e} = \left( \sum_{i=1}^{r-\alpha - 1}
\frac{p-1}{p^{i}}\sigma_{i,x}  \right) +
\frac{1}{p^{r-\alpha - 1}}\sigma_{r - \alpha, x}.$  
Remark \ref{Rupperlower} shows that this is equal to $\sigma^{\alpha}_x$.
\\
\\
\emph{To (iii)}: By Lemma \ref{Ltailetale}, this is (ii) in the case $\alpha =
r' = 0$.
\end{proof}

\begin{lemma}[(Effective local vanishing cycles
formula)]\label{Lgenlocvancycles}
Let $v \in V(\mc{G}')$ correspond to a $p^r$-component $\ol{W}$ of $\ol{X}$ with
genus $g_v$.  Then for all $\alpha <
r$,
\begin{equation}\label{Egenlocvancycles}
\sum_{s(e) = v} (\sigma^{\eff, \alpha}_e - 1) = 2g_v - 2.
\end{equation}
\end{lemma}

\begin{proof}
Each $e \in E(\mc{G}')$ with $s(e) = v$ corresponds to a point $w_e$ on
$\ol{W}$.  
Write $$B = \{e \in E(\mc{G}') \, | \, s(e) = v \text{ and } t(e) 
\text{ does not correspond to a } p^a \text{-component with } a > r\}.$$ 
Write $B' = \{e \in E(\mc{G}') \, | \, s(e) = v\} \backslash B$.  For $e \in B$,
let $\sigma_{i,w_e}$ be the invariant of the $i$th deformation datum above
$\ol{W}$ at $w_e$.  For
$e \in B'$, let $\sigma_{i, w_e}$ be the invariant of the $i$th deformation
above $\ol{W}'$ at $w_e$, where
$\ol{W}'$ is the component that intersects $\ol{W}$ at $w_e$, and
write $n_{w_e}$ as in Lemma \ref{Llocvancycles}.
Then for $1 \leq i \leq r$, Equation (\ref{Elocvancycles}), along with
Proposition \ref{Pannuluscompatibility}, shows that

\begin{equation}\label{Eeff1}
\sum_{e \in B'} \left(-\frac{\sigma_{i+n_{w_e}, w_e}}{p^{n_{w_e}}} - 1 -
\sum_{j=1}^{n_{w_e}}
\frac{p-1}{p^{j}}\sigma_{j,{w_e}} \right) + \sum_{e \in B} (\sigma_{i, w_e} - 1)
= 2g_v - 2.
\end{equation}

For $1 \leq i \leq r- \alpha - 1$, we multiply the $i$th equation (\ref{Eeff1})
by $\frac{p-1}{p^{i}}$ to obtain an equation
$E_i$.  For $i = r - \alpha$, we multiply it by $\frac{1}{p^{r-\alpha - 1}}$ to
obtain $E_{r-\alpha}$.  
Note that these coefficients add up to 1.  Adding up the equations $E_i$ yields
$$\sum_{e \in B'} (-\sigma^{\eff, \alpha}_{\ol{e}} - 1) + \sum_{e \in B}
(\sigma^{\eff, \alpha}_e - 1) = 2g_v - 2.$$
Combining and using Lemma \ref{Lsigmaeffcompatibility} (i) proves the lemma.
\end{proof}

\begin{remark}
Compare (\ref{Egenlocvancycles}) to the vanishing cycles formula in \cite[p.\
998]{We:br}.
\end{remark}

\subsection{Vanishing cycles formulas}\label{Svancycles}
Recall that $m_G = |N_G(P)/Z_G(P)|$, and that we will write $m$ instead of $m_G$
when $G$ is understood.
The vanishing cycles formula (\cite[3.4.2 (5)]{Ra:sp}, \cite[Corollary
1.11]{We:br}) is a key
formula that helps us understand the
structure of the stable reduction of a branched $G$-cover of curves in the case
where $p$ exactly divides the order of $G$. 
Here, we generalize the formula to the case where $G$ has a cyclic $p$-Sylow
group of arbitrary order.
For any \'{e}tale tail $\ol{X}_b$, recall that $\sigma_b$ is the effective
ramification invariant at the
point of intersection $x_b$ of $\ol{X}_b$ with the rest of $\ol{X}$ (Definition
\ref{Draminvariant}). 

\begin{theorem}[(Vanishing cycles formula)] \label{Tvancycles}
Let $f: Y \to X$, $X$ not necessarily $\proj^1$, be a $G$-Galois cover as in
\S\ref{Sstable}, where $G$ has a cyclic
$p$-Sylow subgroup.  As in \S\ref{Sstable}, there is a smooth model $X_R$ of $X$
where the specializations of the branch points do not collide, $f$ has bad
reduction, and $\ol{f}: \ol{Y} \to
\ol{X}$ is the stable reduction of $f$.
Let $\Pi$ be the set of branch points of $f$ that have branching index divisible
by $p$. 
Let $B_{\text{new}}$ be an indexing set for the new \'{e}tale tails and let
$B_{\text{prim}}$ be an indexing set for the primitive
tails.  Let $B_{\text{\'{e}t}} = B_{\text{new}} \cup B_{\text{prim}}$.
Let $g_X$ be the genus of $X$.  Then we have the formula 

\begin{equation}\label{Evancycles2} 
2g_X - 2 + |\Pi| = \sum_{b \in B_{\text{\'{e}t}}} (\sigma_b - 1). 
\end{equation}
\end{theorem}

Theorem \ref{Tvancycles} has the immediate corollary:
\begin{corollary}
Assume further that $f$ is a three-point cover of $\proj^1$.  Then
\begin{equation}\label{Evancycles} 
1 = \sum_{b \in B_{\text{new}}} (\sigma_b - 1) + \sum_{b \in B_{\text{prim}}}
\sigma_b.
\end{equation}
\end{corollary}

\noindent\emph{Proof} (of the theorem, cf. \cite[Corollary 1.11]{We:br}).  
Let $\mc{G}$ (resp. $\mc{G}'$) be the dual graph (resp.\ augmented dual graph)
of the stable reduction of $f$.  For any collection of vertices $\mc{H}
\subseteq V(\mc{G})$ containing $v_0$, 
let $B(\mc{H})$ be the collection of edges $e \in E(\mc{G}')$ such that
$s(e) \in \mc{H}$, but $t(e) \notin \mc{H}$. 
We define $F(\mc{H}) = \sum_{e \in B(\mc{H})} (\sigma^{\eff}_e - 1)$ (see
Definition \ref{Dsigmaeff}).  
Then, if $\mc{H} = \{v_0\}$, $F(\mc{H}) = 2g_X - 2$ by (\ref{Elocvancycles}). 
By outward induction,
using Lemmas \ref{Lgenlocvancycles} and \ref{Lsigmaeffcompatibility} (i), we can
add one adjacent
vertex at a time to $\mc{H}$ without changing the value of $F(\mc{H})$, so long
as the vertex corresponds to an
inseparable component.  Thus, if $\mc{H} = V(\mc{G}_0)$ (p.\ \pageref{TEgraph}),
then  $F(\mc{H}) = 2g_X - 2$.  
By Lemma \ref{Lsigmaeffcompatibility} (iii), we obtain $\sum_{b \in
B_{\text{\'{e}t}}} (\sigma_b - 1) - |\Pi| = 2g_X - 2$.
\qed

\begin{remark}
One can also construct a proof analogous to that of \cite[3.4.2 (5)]{Ra:sp},
using the \emph{auxiliary
cover}.  This is done in the author's thesis \cite[\S 3.1]{Ob:th}.
\end{remark}

The above formula can be generalized.
For every $i$, $1 \leq i \leq n$, write $\Pi_i$ for the set of branch points of
$f$
which have branching index divisible by $p^i$.  
Let $B_{r,r'},$ $r > r'$, be as in Definition \ref{Draminvariant}, each $b \in
B_{r, r'}$ corresponding to a point $x_b$.
Then we have the following formula:

\begin{prop}\label{Pgenvancycles}
Fix $\alpha$ such that $0 \leq \alpha \leq n-1$ and there exists some nonempty
$B_{r,r'}$
with $r' \leq \alpha < r$.  Let $B^{\alpha}_{r, r'} \subset B_{r, r'}$ be the
subset consisting of those $b$ such that the 
vertex corresponding to the $p^r$-component containing $x_b$ is a maximal vertex
for $\mc{G}_{\alpha}$ 
(with ordering induced from $\mc{G}'$, see p.\ \pageref{TEgraph}).  Then
\begin{equation} \label{Egenvancycles}
2g_X - 2 + |\Pi_{\alpha+1}| \geq \sum_{\substack{r, r' \\ r' \leq \alpha < r}}
\sum_{b \in B^{\alpha}_{r,r'}}
(\sigma^{\alpha}_b - 1).
\end{equation}
If $\ol{f}$ is monotonic, we have equality in (\ref{Egenvancycles}).
\end{prop}

\begin{proof}  
Let $\mc{U}_i$, $i \in I$, be the set of connected components of
$\mc{G}_{\alpha}$.  
For each $i \in I$, let $B(\mc{U}_i) \subset E(\mc{G}')$ be the set of those
edges $e$ such
that $s(e) \in V(\mc{U}_i)$ but $t(e) \notin V(\mc{U}_i)$.  Define $F(\mc{U}_i)
= \sum_{e \in B(\mc{U}_i)}
(\sigma^{\eff, \alpha}_e - 1)$.  Then, as in the proof of Theorem
\ref{Tvancycles}, we have $F(\mc{U}_i) = 2g_X - 2$ if
$v_0 \in \mc{U}_i$, and $F(\mc{U}_i) = -2$ otherwise.  

Set $\delta = 1$ if $v_0 \in V(\mc{G}_{\alpha})$, and $\delta = 0$ otherwise. 
Then we know 
\begin{equation}\label{Evancyclesineq}
\sum_{i \in I} \sum_{e \in B_{\mc{U}_i}} (\sigma^{\eff, \alpha}_e - 1) = 
\sum_{i \in I} F(\mc{U}_i) = -2|I| + 2\delta
g_X.
\end{equation}
Lemma \ref{Lsigmaeffcompatibility} (ii) shows that, for $e \in B(\mc{U}_i)$
corresponding to $b$ in some 
$B^{\alpha}_{r, r'}$, we have $\sigma^{\eff, \alpha}_e = \sigma^{\alpha}_b$.
In any case, for $e \in B(\mc{U}_i)$, Lemma \ref{Lsigmaeffcompatibility}
(ii)--(iii) shows that 
$\sigma^{\eff, \alpha}_e \geq 0$, with equality holding iff 
$t(e) \in \mc{G}' \backslash \mc{G}$.  Also, an easy combinatorial argument
shows that 
$$|\bigcup_{i \in I} B(\mc{U}_i)| \ - \ |\bigcup_{r' \leq \alpha < r}
B^{\alpha}_{r, r'}| = |\Pi_{\alpha + 1}| + 2|I| - 2.$$
We expand out $F(\mc{U}_i)$ in (\ref{Evancyclesineq}), using the inequality
$\sigma^{\eff, \alpha}_e - 1 \geq -1$ for those 
$e \in B(\mc{U}_i)$ not corresponding to elements of $B^{\alpha}_{r, r'}$.  This
yields
$$\left(\sum_{\substack{r, r' \\ r' \leq \alpha < r}}\sum_{b \in
B_{r,r'}^{\alpha}} (\sigma^{\alpha}_b - 1)\right) - 
(|\Pi_{\alpha + 1}| + 2|I| - 2) \leq 
-2|I| + 2g_X,$$ with equality iff $\delta = |I| = 1$, i.e., $\mc{G}'_{\alpha}$
is connected.  A simple rearrangement
yields Equation (\ref{Egenvancycles}).  

Lastly, if $\ol{f}$ is monotonic, then clearly $\delta = |I| = 1$, so we have
equality.
\end{proof} 

\begin{remark}
The case $\alpha = 0$ of the generalized vanishing cycles formula
(\ref{Egenvancycles}) is the vanishing cycles formula
(\ref{Evancycles2}).
\end{remark}

\section{Properties of tails of the stable reduction}\label{Scomb} 

We maintain the assumptions of \S\ref{Sstable}, along with the assumption that a
$p$-Sylow subgroup of $G$ is cyclic of
order $p^n$.  Throughout, we will use the abbreviation $m = m_G$, as well as the
notations $B_{r,r'}$ and the variants on 
$\sigma^{\alpha}_{x_b}$ from Definition \ref{Draminvariant}. 

\begin{lemma}\label{Linsepint}
If $b \in B_{r, r'}$ indexes an inseparable tail $\ol{X}_b$ (so $r' > 0$), 
then all $\sigma^{\alpha}_b$'s $(r' \leq \alpha < r)$ are integers.  In
particular, $\sigma_b \in \ints$.
\end{lemma}
\begin{proof} 
Consider an irreducible component $\ol{Y}_b$ of $\ol{Y}$ lying above
$\ol{X}_b$.  By Corollary \ref{Cinsepnormalp}, its decomposition group
$D_{\ol{Y}_b}$ has a normal subgroup of order
$p$.  Thus, by Corollary \ref{Cnormalp} (i), there exists $N < D_{\ol{Y}_b}$
with $p \nmid |N|$ such that $D_{\ol{Y}_b}/N$ 
is of the form $\ints/p^a \rtimes \ints/\ell$ (for some $a > r'$ and $\ell$ with
$p \nmid \ell$).
Now, by Remark \ref{Rcorrectspec}, if $y$ is a ramification point of $f$
specializing to $\ol{y}$ on $\ol{Y}_b$, then
the inertia group of $\ol{y}$ in $\ol{Y}_b \to \ol{X}_b$ is cyclic of order
$p^{r'}s$, with $p \nmid s$.  
Then the inertia group of the image of $\ol{y}$ in $\ol{Y}_b/N \to \ol{X}_b$ is
also cyclic.  Since $p \nmid |N|$, the
order of this inertia group must be $p^{r'}$.

Let $H$ be the unique subgroup of $\Aut((\ol{Y}_b/N)/\ol{X}_b)$ of order $p^a$. 
Then, by Proposition \ref{Pspecialram},
the map $(\ol{Y}_b/N)/H \to \ol{X}_b$ can only be ramified at the intersection
point $x_b$ of $\ol{X}_b$ 
with the rest of $\ol{X}$.  Thus it is a trivial cover, and we see that $\ell =
1$.  So 
$\ol{Y}_b/N \to \ol{X}_b$ is a cyclic cover of
order $p^a$, branched at $x_b$.  By the Hasse-Arf Theorem (\cite[V, Theorem
1]{Se:lf}), the upper jumps of higher
ramification at $x_b$ are all integral.  Since the upper numbering is invariant
under quotients, we conclude that
$\sigma^{\alpha}_b$ is integral for all $r' \leq \alpha < r$.
\end{proof}

\begin{lemma} \label{Ltailbounds}
\begin{description}
\item{(i)} A new tail $\ol{X}_b$ (\'{e}tale or inseparable) has $\sigma_b \geq 1
+ 1/m$.  
\item{(ii)} If $r' > 0$, and $\ol{X}_c$ is any $p^{r'}$-tail that borders a
$p^r$-component,  
then $\sigma_b \geq p^{r-r'-1}$.
\item{(iii)} If $\ol{X}_b$ is a primitive \'{e}tale tail that borders a
$p^r$-component, then
$\sigma_b \geq p^{r-1}/m$.
\end{description}
\end{lemma}

\begin{proof} For part (i), assume that $\ol{X}_b$ is a $p^{r'}$-component.
If we assume that (i) is false, then either $\ol{X}_b$ borders a $p^r$-component
with $r-r' \geq 2$ or
\cite[Lemme 1.1.6]{Ra:sp} shows that each irreducible component
above $\ol{X}_b$ is a genus zero Artin-Schreier cover of $\ol{X}_b$.  In the
first case, 
we cite \cite[Lemma 19]{Pr:lg}, which shows that $\sigma^{\alpha}_b \geq 
p\sigma^{\alpha+1}_b$ for all $\alpha$ where the statement makes sense.  Then
$\sigma_b := \sigma^{r'}_b \geq
p^{r-r'-1}\sigma^{r-1}_b$.  By Lemma \ref{Lhassearf}, $\sigma^{r-1}_b \geq
\frac{1}{m}$, so 
$p^{r-r'-1}\sigma^{a-1}_b \geq \frac{p}{m} \geq 1 + 1/m,$ as $m | (p-1)$.
In the second case, since no ramification points specialize to the components
over $\ol{X}_b$, 
this contradicts the three-point condition of the stable model.  

For parts (ii) and (iii), we have $\sigma_b = \sigma^{r'}_b$.
If $\ol{X}_b$ is inseparable, then $1 \leq \sigma^{r-1}_b \in \ints$ (by Lemma
\ref{Linsepint}).  So $\sigma^{r'}_b \geq p^{r - r' - 1}$ (again, using
\cite[Lemma 19]{Pr:lg}), proving (ii).
Also, $\sigma^{r-1}_{b} \geq 1/m$ for $\ol{X}_b$ \'{e}tale and primitive, by
Lemma \ref{Lramdenominator}.  
Then (iii) follows using \cite[Lemma 19]{Pr:lg}.
\end{proof}

\begin{remark}\label{Rvancyclespositive}
Lemma \ref{Ltailbounds} (i) and (iii) show that all terms on the right-hand side
of the vanishing cycles formula
(\ref{Evancycles}) are positive.
\end{remark}

We now give some sufficient criteria for the stable reduction of $f$ to be
monotonic. 
\begin{prop} \label{Pmonotonic}
For any $G$, if $\ol{T}$ is a component of $\ol{X}$ such that there 
are no \'{e}tale tails $\ol{X}_b \succ \ol{T}$, then $\ol{f}$ is monotonic from
$\ol{T}$.  
\end{prop}
\begin{proof} 
Suppose $\ol{T}$ is a $p^i$-component, and there are no \'{e}tale 
tails lying outward from $\ol{T}$.
For a contradiction, suppose $t \in \ol{T}$ is a point such that there exists a
$p^j$-component $\ol{W}$, 
with $j > i$, lying outward from $t$.
Consider the morphism $X^{st} \to X'$ that is ``the identity" on $X_{K^{st}}$
and contracts $\ol{U}$, the union of all
components of $\ol{X}$ outward from $t$.  
If $Y'$ is the normalization of $X'$ in $K^{st}(Y)$, then $Y'$ is obtained from
$Y^{st}$ by
contracting all of the components of the special fiber above those components
contracted by $X^{st} \to X'$.
Let $y$ be a point of $Y'$ lying over the image of $t$ in $X'$ (which we call
$t$, by abuse of notation), and consider the
map of complete germs $\hat{Y}'_y \to \hat{X}'_t$.  
This map is Galois.  Its Galois group $G'$ contains the decomposition group of
the connected component of
$\ol{f}^{-1}(\ol{U})$ containing a preimage of $y$ in $\ol{Y}$.  By Corollaries
\ref{Cnormalp} (i) and \ref{Cinsepnormalp}, 
there exists $N \leq G'$ such that $G'/N \cong \ints/p^a \rtimes \ints/\ell$,
with $p \nmid \ell$ (this is the only place we use the assumption that there are
no \'{e}tale tails lying outward from $\ol{T}$).  
Also, $a \geq j > i$ by our assumption on $t$.

Let $\hat{V}_v$ be the quotient of $\hat{Y}'_y$ by the subgroup of $G'$ that
contains $N$ and whose image in $G/N$ has order $p^i$.  Then 
$\phi: \hat{V}_v \to \hat{X}'_{t}$ is Galois with Galois group $\ints/p^{a-i}
\rtimes \ints/\ell$.  
Note that $t$ is a \emph{smooth} point 
of $X'$, as we have contracted only a tree of projective lines (of course, $y$
may be quite singular, but it is still a normal point of $Y'$).

Now, $\phi$ is totally ramified above the point $t$, but it is
unramified above the height 1 prime $(\pi)$, where $\pi$ is a uniformizer of
$R$, because we have quotiented out the
generic inertia of $\ol{T}$.  Using purity of the branch locus (\cite[Theorem
5.2.13]{Sz:fg}), we see that
$\phi$ must be ramified over some height 1 prime $(u)$ such that the scheme cut
out by $u$ intersects the generic fiber.  Since we have been assuming from the
beginning that the branch
points of $Y_{K} \to X_{K}$ do not collide on the special fiber $\ol{X}_0$,
and we have not contracted $\ol{X}_0$,
there is at most one branch point on the generic fiber that can
specialize to $t$.  Thus $(u)$ cuts the generic fiber in exactly one point,
and it is the only height 1 prime above which $\phi$ is ramified.  So $\phi$ is
\'{e}tale outside of the scheme cut out
by $(u)$.  We are now in the situation of \cite[Lemme 6.3.2]{Ra:ab}, and we
conclude that 
the ramification index at $(u)$ is prime to $p$.  But this contradicts the fact
that the ramification index above $t$ is
divisible by $p^{a-i}$.  
\end{proof}

\begin{remark}\label{Rmonotonic}
The proof above shows that, if $G \cong \ints/p^n \rtimes \ints/m$, then
$\ol{f}$ is monotonic from $\ol{T}$, even if there are 
\'{e}tale tails lying outward from $\ol{T}$.
\end{remark}

For the rest of this section, assume that $f:Y \to X$ is a three-point cover of
$\proj^1$ with bad reduction.  

\begin{prop}\label{Ptaillimit}
The stable reduction $\ol{X}$ has fewer than $p$ \'{e}tale tails.  For $d \geq
1$, the number of $p^d$-tails of $\ol{X}$
is less than $p^d$.
\end{prop}
\begin{proof} 
We proceed by strong induction on $d$, proving the base cases $d=0$ (i.e., the
\'{e}tale tails) and $d=1$
separately.  

For $d = 0$, Equation (\ref{Evancycles}) gives 
$$1 = \sum_{b \in B_{\text{new}}} (\sigma_b - 1) + \sum_{b \in B_{\text{prim}}}
\sigma_b.$$
Each term on the right-hand side above is at least $1/m$, by Lemma
\ref{Ltailbounds}.  
So there are at most $m$ \'{e}tale tails.  Since $m | (p-1)$, the case $d=0$ is
proved.

For $d=1$, consider (\ref{Egenvancycles}) for $\alpha = 1$.
In the notation of (\ref{Egenvancycles}), let $B' \subset \bigcup_{r>1} B_{r,1}$
be the set of tails which are 
$p$-components and which contain the specialization of a branch point of $f$.  
Then we obtain
\begin{equation}\label{Eptails}
1 \geq \sum_{b \in \bigcup_{r>1} B^1_{r,1} \backslash B'} (\sigma_{b} - 1) +
\sum_{b \in B'} (\sigma_{b}).
\end{equation}
We note that the number of points indexed by $\bigcup_{r > 1} B^1_{r,1}$ that do
\emph{not} lie on a $p$-tail is bounded
by the number of \'{e}tale components, i.e., there can be no more than $p-1$ of
them (this is because each such point
lies on a $p$-component that has an \'{e}tale tail lying outward from it, and
two such $p$-components do not share the 
same \'{e}tale tail).  On the right-hand side of (\ref{Eptails}), each term
corresponding to such a
$p$-component contributes at least $\frac{1}{m} - 1$, thus at least
$\frac{1}{p-1} - 1 = \frac{2-p}{p-1}$.
By Lemma \ref{Ltailbounds} (ii), each other term on the right-hand side of
(\ref{Eptails}) is nonnegative.  Also, each tail which is a $p$-component 
corresponds to a term on the right-hand side of (\ref{Eptails}), and each such
term is at least 1, by Lemma
\ref{Ltailbounds}.  
Thus the right-hand side is at least $\frac{2-p}{p-1}(p-1) + \nu$, 
where $\nu$ is the number of tails which are $p$-components.  We conclude from
(\ref{Eptails}) that $\nu \leq p - 1$,
proving the case $d =1$.

Now, assume the lemma holds up through $d = \delta$.  
The number of $p^{\delta+1}$-components $\bigcup_{r > \delta}
B^{\delta}_{r,\delta}$ which are not tails is bounded by the
number of tails which are $p^a$-components for some $a \leq \delta$.  
By the inductive hypothesis, this is bounded by $$M = (p-1) + (p-1) + (p^2-1) +
\cdots +
(p^{\delta}-1) = \frac{p^{\delta +1}-1}{p-1} + p
- \delta - 2.$$ 
Some calculation shows that this is less than $(p^{\delta+1} -
1)\frac{p-1}{p-2}$ for $p > 2$.

Analogously to the case of $d=1$, Equation (\ref{Egenvancycles}) for $\alpha =
\delta + 1$ yields the inequality
$$1 \geq \frac{2-p}{p-1}M + \nu > 1 -p^{\delta +1}  + \nu,$$ where $\nu$ is the
number of tails which are $p^{\delta+1}$-components (if $p=2$, the second
inequality holds without any condition
on $M$).  We conclude that $\nu < p^{\delta+1}$.  
\end{proof}

\begin{corollary}\label{Ctaillimit}
For any $\mu \in \reals$, $d > 0$, let $S_{d, \mu}$ be the set of $p^d$-tails
$\ol{X}_b$ that satisfy 
$\sigma_b - 1 \geq p^{\mu}$.  Then the cardinality of $S_{d, \mu}$ is less than
$p^{d-\mu}$.
\end{corollary}

\begin{proof}
We carry the proof of Proposition \ref{Ptaillimit} through.  In the $d = 1$
step, if we let $\nu = |S_{1, \mu}|$ (as
opposed to the \emph{total} number of $p$-tails), then we obtain from
(\ref{Eptails}) that $p^{\mu}\nu \leq p-1$.  
This gives the corollary for $d=1$.  In the inductive
step, we set $\nu = |S_{\delta+1, \mu}|$ as opposed to the total number of
$p^{\delta+1}$-tails.  
Again, we conclude that $p^{\mu}\nu < p^{\delta+1}$, which gives the corollary.
\end{proof}

The following proposition will be useful for Example \ref{Xgoodreduction}:

\begin{prop}\label{Pnoinsepnonew}
Suppose $\ol{X}$ has no new \'{e}tale tails.  Then it has no new inseparable tails.
\end{prop}

\begin{proof}
Assume, for a contradiction, that $i \geq 1$ is minimal such that there is a new $p^i$-tail $\ol{X}_b$.  
Applying (\ref{Egenvancycles}) for $\alpha = i$ yields (in the notation of (\ref{Egenvancycles}))
$$- 2 + |\Pi_{i+1}| \geq \sum_{\substack{r, r' \\ r' \leq i < r}} \sum_{b \in B^{i}_{r,r'}} (\sigma^{i}_b - 1).$$
On the right-hand side, using Proposition \ref{Pmonotonic}, all terms $b$ that do not correspond to $p^i$-tails correspond to points with
primitive \'{e}tale tails lying outward from them.  There are at most $3 - |\Pi_{i+1}|$ such tails (and at least $1$, by the vanishing cycles formula
(\ref{Evancycles})), so there are at most $3 - |\Pi_{i+1}|$ such
$b$, for which $\sigma^i_b - 1 > -1$ in each case.   By Lemmas \ref{Linsepint} and \ref{Ltailbounds}, the $b$ corresponding to new $p^i$-tails 
(of which there is at least $1$)
satisfy $\sigma^i_b - 1 \geq 1$ and those corresponding to other $p^i$-tails satisfy $\sigma^i_b - 1 \geq 0$.  Putting this all together, we 
see that the
right-hand side is \emph{strictly} greater than $-(3 - |\Pi_{i+1}|) + 1 = -2 + |\Pi_{i+1}|$.  This is a contradiction.
\end{proof}

\section{Wild Monodromy and Stable Reduction}\label{Swild}

\subsection{The main theorem}
We maintain the assumptions and notations of \S\ref{Sstable}.
In particular, $G$ is a finite group with cyclic $p$-Sylow subgroup $P$ of order
$p^n$, and
$m_G =  |N_G(P)/Z_G(P)|$.  Assume $n \ne 0$ (so that $p$
divides the order of $G$).  We make the additional (important!) assumption that
$p$ \emph{does not divide the order 
of the center of $G$}.  Let $K_0 = \Frac(W(k))$, where $k$ is an algebraically
closed field of characteristic $p$.  
Let $f: Y \to X = \proj^1$ be a three-point $G$-cover 
defined over a finite extension $K/K_0$.  Write $K^{st}/K$ for the
smallest extension of $K$ over which 
the stable model of $f$ can be defined.  Then $\Gamma := \Gal(K^{st}/K)$ is
called the \emph{monodromy group}, and its
(unique) $p$-Sylow subgroup $\Gamma_w$ is called the \emph{wild monodromy
group}.  
Recall from \S\ref{Sstable} that $\Gamma$ is the largest quotient of $G_K$ that
acts faithfully on the stable reduction $\ol{f}: \ol{Y} \to \ol{X}$ of $f$.  
So $\Gamma$ acts on $\ol{Y}$ and the action descends to an action on $\ol{X}$.
Furthermore, the action of $\Gamma$ commutes with the action of $G$.  Theorem \ref{Tmain} states that $\Gamma_w$ has exponent dividing
$p^{n-1}$.  In other words, for any $g \in \Gamma_w$, $g^{p^{n-1}} = 1$.  

The proof of Theorem \ref{Tmain}, which is spread out over \S\ref{Sproof}, relies on methods
similar to those used in \cite{Ra:sp}.  The main idea is to
examine possible $p$-power order actions on the 
stable reduction of $f$ in detail, and to show that actions of order $p^n$
cannot be induced by elements of $\Gamma_w$.  
Our main tools are the generalized vanishing cycles formula
(\ref{Egenvancycles}) and Proposition \ref{Ptaillimit}.

\subsection{Proof of Theorem \ref{Tmain}}\label{Sproof}

Assume first that $f$ has potentially good reduction.  
Then, by \cite[Proposition 4.2.2]{Ra:sp}, the wild monodromy group $\Gamma_w$ 
is isomorphic to a subgroup of the $p$-Sylow subgroup $Q$ of the center of $G$,
which is trivial by assumption.

Now assume that $f$ does not have potentially good reduction.
We will use the notations $B_{r, r'}$, $J_x$ as well as $\sigma^{\alpha}_{x_b}$
and its variants from Notation \ref{Nsx}
and Definition \ref{Draminvariant}.

We study how $\Gamma_w$ acts on different parts of $\ol{Y}$ and $\ol{X}$.
We start with an easy lemma:

\begin{lemma}\label{Ltwopoints}
If $\gamma \in \Gamma_w$ acts on a component $\ol{W}$ of $\ol{X}$, then it fixes
pointwise any component $\ol{W}' \prec
\ol{W}$ of $\ol{X}$.
\end{lemma}

\begin{proof}
Since $k$ is algebraically closed, all elements of $G_K$ commute with the
reduction from $R$ to $k$.  Thus $\Gamma_w$, which is a subquotient of $G_K$,
acts trivially on $\ol{X}_0$, which is the 
reduction of the standard model of $\proj^1_R$ to $k$.  So we may assume
$\ol{W}' \ne \ol{X}_0$.
By continuity, $\gamma$ fixes the singular points of $\ol{X}$ lying on $\ol{W}'$
in the directions of $\ol{X}_0$ and 
$\ol{W}$.  Then $\gamma$ acts on $\ol{W}' \cong \proj^1$ with at least two fixed
points and $p$-power order.  So $\gamma$ acts trivially.
\end{proof}

We will look separately at the action of $\gamma$ on the \'{e}tale tails, and
then on the inseparable tails.

\subsubsection{The \'{e}tale tails.}
We first examine how the action of $\Gamma_w$ interacts with the \'{e}tale
tails.
\begin{lemma} \label{Lfixetale}
The action of $\Gamma_w$ permutes the \'{e}tale tails of $\ol{X}$ trivially.
\end{lemma}
\begin{proof}
(cf.\ \cite[p.\ 112]{Ra:sp}) By Proposition \ref{Ptaillimit}, there are at most
$p-1$ \'{e}tale
tails.  But $\Gamma_w$ is a $p$-group, and thus must permute them trivially.
\end{proof}

\begin{corollary} \label{Cfixbottom}
The group $\Gamma_w$ acts trivially on the components $\ol{W}$ of $\ol{X}$ such
that there exists an \'{e}tale tail
$\ol{X}_b \succ \ol{W}$.
\end{corollary}
\begin{proof}
This follows from Lemmas \ref{Ltwopoints} and \ref{Lfixetale}.
\end{proof}

\begin{lemma} \label{Lvertactioncentral}
Let $\ol{Y}_b$ be a component of $\ol{Y}$ lying above an \'{e}tale tail
$\ol{X}_b$ of $\ol{X}$.  If $\gamma \in \Gamma_w$ acts trivially on $\ol{X}_b$ and acts on $\ol{Y}_b$, then it acts trivially on
$\ol{Y}_b$.
\end{lemma}

\begin{proof}
Write $D := D_{\ol{Y}_b}$, the decomposition group of $\ol{Y}_b \to \ol{X}_b$. 
Suppose that $\gamma$ acts nontrivially on $\ol{Y}_b$.  Then we can view $\gamma$ as an element of the
center of $D$, and replacing $\gamma$ with some power, we may assume that $\gamma$ has order $p$ in $D$.   
Let $Q \leq D$ be the central subgroup of order $p$ generated by $\gamma$.  We will show that $Q$ is central in $G$, contradicting
the running assumption of this section.

Corollary \ref{Cnormalp} (ii) shows that $m_D = 1$.  By Lemma \ref{Lramdenominator}, we have
$\sigma_b \in \ints$ for $b$ corresponding to $\ol{X}_b$.  By the vanishing cycles formula (\ref{Evancycles}),
we have $\sigma_b = 1$ if $\ol{X}_b$ is primitive and $\sigma_b = 2$ if $\ol{X}_b$ is new.
By Lemma \ref{Ltailbounds}, $\ol{X}_b$ borders a $p$-component, unless $p = 2$, in which case $\ol{X}_b$ can border a
$p^2$-component if it is new.  In all cases, (\ref{Evancycles}) shows that $\ol{X}_b$ is the only \'{e}tale tail.

We claim that $\ol{Y}_b \to \ol{X}_b$ is totally ramified above the singular point $\ol{x}_b$ of $\ol{X}$ on $\ol{x}_b$.  To prove the claim,
assume first that $\ol{X}_b$ borders a $p$-component.  Since $Q$ is central in $D$, 
it is contained in every $p$-subgroup of $D$, in particular, every
wild inertia group of a wildly ramified point of $\ol{Y}_b$.  By Proposition \ref{Pspecialram} (ii), $Q$ is the 
$p$-Sylow subgroup of all these inertia groups, and thus $h: \ol{Y}_b/Q \to \ol{X}_b$ is tamely ramified, branched
at at most two points.  So $h$ is a cyclic, totally ramified cover of degree $|D/Q|$, and $\ol{Y}_b/Q$ has genus zero.  
Therefore, $\ol{Y}_b \to \ol{Y}_b/Q$ is an Artin-Schreier cover, totally ramified at its unique ramification point, proving the claim in this case.

Now, assume $\ol{X}_b$ borders a $p^2$-component (so $p=2$ and $\ol{X}_b$ is new).  By \cite[Lemma 19]{Pr:lg}, $\sigma^1_b = 1$.
Again, $Q$ is contained in every inertia group, so $h: \ol{Y}_b/Q \to \ol{X}_b$ is a $D/Q$-cover with
inertia groups with $2$-Sylow subgroup of order $2$.  But $m_{D/Q} = 1$, so $D/Q$ has a subgroup $N$ of odd order such that
$|(D/Q)/N|$ is cyclic of $2$-power order (Lemma \ref{Lburnside}).  Since the inertia groups of $j: (\ol{Y}_b/Q)/N) \to \ol{X}_b$ have order $2$, we 
must have that $|(D/Q)/N| = 2$, and that $j$ is totally ramified above $\ol{x}_b$.  Since $\sigma_b^1 = 1$, a Hurwitz formula 
calculation shows that the genus of $(\ol{Y}_b/Q)/N$, being branched at only one point with conductor 1, is zero.  This implies that $N$ is trivial,
as $\ol{Y}_b/Q \to (\ol{Y}_b/Q)/N$ is a prime-to-$2$ cover of $\proj^1$ branched at one point.  So $|D| = 4$, and we have total ramification above
the singular point of $\ol{X}$.  The claim is proved.

Our next claim is that $Q$ is normal in $G$.  To show this, we note that because $\ol{Y}_b \to \ol{X}_b$ is totally ramified above $\ol{x}_b$, 
then $\ol{Y} \backslash \ol{Y}_b$ is still connected.  Clearly, we may remove all of the other components above $\ol{X}_b$ while preserving this
connectivity.  Since $G$ acts on what remains of $\ol{Y}$, and all of the remaining components have nontrivial inertia, then Corollary 
\ref{Cinsepnormalp} shows that $G$ has a normal subgroup of order $p$.  But since $G$ has cyclic $p$-Sylow subgroup, it can only have one subgroup
of order $p$, which must be $Q$.  The claim is proved.

Lastly, we show that $Q$ is central in $G$.  Now, since $\ol{X}_b$ is the only \'{e}tale tail, 
there is at most one branch point of $f: Y \to X$ with prime-to-$p$ branching index (it exists iff $\ol{X}_b$ is primitive).  
By Corollary \ref{Cnormalp} (i), there exists a subgroup $N \leq G$ such that $p \nmid |N|$ and $G/N \cong \ints/p^n \rtimes \ints/m_G$, 
with faithful action.  
So $Y/N \to X$ also has at most one branch point with prime-to-$p$ branching index.   Let $P \leq G/N$ be the unique $p$-Sylow subgroup.
Then $(Y/N)/P \to X$ is a $\ints/m_G$-Galois cover, tamely ramified at at most one point.  Thus it is trivial, $m_G = 1$, and $Q$ maps
isomorphically onto its image $Q/N$ in $G/N \cong \ints/p^n$.  The conjugation action of $G$ on $G/N$ is well defined, and clearly trivial.
Hence, the action of $G$ on $Q/N$ is trivial, so the action of $G$ on $Q$ is trivial.  This completes the proof of the lemma.
\end{proof}

\begin{lemma} \label{Lfixetalevert}
Let $\ol{X}_b$ be an \'{e}tale tail of $\ol{X}$ intersecting the rest of
$\ol{X}$ at $x_b$,
and let $\ol{Y}_b$ be a component of $\ol{Y}$ lying above $\ol{X}_b$.  Write
$\Sigma$ for the set of singular points of
$\ol{Y}$ lying on $\ol{Y}_b$.  Suppose that $\gamma \in \Gamma_w$ acts
nontrivially on $\ol{Y}_b$.  
Then $\gamma$ acts with order $p$ on $\ol{Y}_b$, and fixes no point of $\Sigma$.
\end{lemma}

\begin{proof}   
By Lemma \ref{Lvertactioncentral}, we see that  the action of $\gamma$ on $\ol{Y}_b$ descends faithfully to
$\ol{X}_b$.  If $\ol{X}_b$ is a primitive tail, then $\gamma$ fixes two points on the tail
and thus acts trivially, so we may
assume $\ol{X}_b$ is a new tail.  Since $\ol{X}_b \cong \proj^1$,
the action of $\gamma$ on $\ol{X}_b$ has order $p$.  Now, suppose $\gamma$ fixes
a point $y_b \in \ol{Y}_b$ that is a
singular point of $\ol{Y}$, lying above $x_b \in \ol{X}$.  
Consider the cover $\ol{Y}'_b \to \ol{X}'_b$ given by quotienting $\ol{Y}_b \to
\ol{X}_b$ by the group generated by 
$\gamma$.  Let $\sigma'_b$ be the corresponding effective ramification invariant
for $\ol{Y}'_b
\to \ol{X}'_b$.  Since $\ol{X}_b \to \ol{X}'_b$ is an Artin-Schreier extension
with conductor 1, Lemma \ref{Lcompositum} (with $\tau = 1$)
shows that $\sigma_b - 1 = p(\sigma'_b - 1)$.  But $\sigma'_b \ne 1$, because if
$\sigma'_b = 1$ then $\sigma_b = 1$, and
we know $\sigma_b > 1$ (Lemma \ref{Ltailbounds}).
So $\sigma'_b > 1$ (\cite[Lemme 1.1.6]{Ra:sp}), and in particular, $\sigma'_b
\geq 1 + \frac{1}{m} \geq 1 + \frac{1}{p-1}$ 
(Lemmas \ref{Lhassearf} and \ref{Lautaction}).
So $\sigma_b -1 = p(\sigma'_b - 1) \geq \frac{p}{p-1} > 1$, which contradicts
the vanishing cycles formula 
(\ref{Evancycles}).  
\end{proof}

\subsubsection{The inseparable tails.}

Suppose $\ol{X}_b$ is an inseparable tail whose intersection $x_b$ with the rest
of $\ol{X}$ is in $B_{r, r'}$ 
(by Lemma \ref{Ltailetale}, $r > r'$).
Pick any component $\ol{Y}_b$ lying above $\ol{X}_b$.
Let $y_b$ be a point of intersection of $\ol{Y}_b$ with the rest of $\ol{Y}$.  
Write $Q$ for a $p$-Sylow subgroup of $D_{\ol{Y}_b}/I_{\ol{Y}_b}$.  It is
cyclic.

Since no branch point of the generic fiber $Y \to X$ specializes to
$\ol{X}_b$, the map $\ol{Y}_b \to \ol{X}_b$ is branched only at $x_b$.
Now let $D$ be the maximal subquotient of $\Gamma_w$ which fixes $y_b$ and acts
faithfully on $\ol{Y}_b$.  If $D$ contains an element $\gamma$ such that $\gamma^p$ generates $Q$, then we set 
$\epsilon_b = 1$ (this implies that $\ol{X}_b$ is a new inseparable tail, because otherwise any $\gamma \in D$ fixes two points on $\ol{X}_b$, 
thus all of $\ol{X}_b$).  Otherwise, we set
$\epsilon_b = 0$.  Note that, since any action on $\ol{X}_b$ of $p$-power order
has order dividing $p$, the maximal
possible order of any element of $D$ is $p^{\epsilon_b}|Q|$.
The main lemma we need is the following:

\begin{lemma} \label{Lsigmaaction}
If $\epsilon_b = 1$, then $\sigma_b - 1 \geq p$.  
\end{lemma}
\begin{proof} 
Suppose that $\epsilon_b = 1$.  We will show that $\sigma^{r - 1}_b - 1 \geq p$.
 Since $r' \geq 1$, it then follows that 
$\sigma_b -1 = \sigma^{r-r'}_b -1 \geq \sigma^{r - 1}_b -1 \geq p$.  

Let $Q' \leq Q$ denote the unique subgroup of index $p$.  Then $Q'$ is central
in $D$.
Write $\ol{V}_b = \ol{Y}_b/Q'$, and let $v_b$ be the image of $y_b$.  
Then the canonical map $\phi: \ol{V}_b \to \ol{X}_b$ has the same first upper
jump at $v_b$ as the map 
$\Phi: \ol{Y}_b \to \ol{X}_b$ does at $y_b$, because the upper numbering is
preserved by
quotients.  Also, $D/Q'$ acts faithfully on $\ol{V}_b$ and commutes with the map
to
$\ol{X}_b$.  Viewed as an element of $D/Q'$, $\gamma$ has order $p^2$.

Let $q: \ol{X}_b \to \ol{X}'_b$ be the quotient of $\ol{X}_b$ by the action of
the group generated by $\gamma$, and let 
$x'_b$ be the image of $x_b$.  
We restrict our attention to the formal neighborhoods 
$\hat{v}_b$, $\hat{x}_b$, and $\hat{x}'_b$ of $v_b$, $x_b$, and $x'_b$ in
$\ol{V}_b$, $\ol{X}_b$, and $\ol{X}'_b$, 
respectively:
The map of germs $\hat{v}_b \to \hat{x}'_b$ has Galois group
$\ints/p^2 \times H$, where $H$ is of prime-to-$p$ order.  
Since the upper numbering is preserved by quotients, we can quotient out by $H$
and assume $H$ is trivial.
Our situation is summarized by the following diagram:

$$\overbrace{\hat{v}_b \underbrace{\stackrel{\phi}{\to}}_{\ints/p} \hat{x}_b 
\underbrace{\stackrel{q}{\to}}_{\ints/p} \hat{x}'_b}^{\ints/p^2}$$

The \emph{second lower jump} of wild ramification of $\hat{v}_b \to
\hat{x}'_b$ is equal to the 
\emph{lower jump} of wild ramification of $\hat{v}_b \to \hat{x}_b$, because the
lower numbering respects subgroups.  
Since the wild ramification of $\hat{v}_b \to \hat{x}_b$ is of order $p$,
this is the same as the upper jump, which is what we wish to calculate.  
Now, the first upper jump of $\hat{v}_b \to \hat{x}'_b$ is 1, because $\hat{X}_b
\to \hat{X}'_b$ is an Artin-Schreier cover of conductor 1.
By \cite[Lemma 4.2]{Pr:lg}, the second \emph{upper} jump of wild ramification of
$\hat{v}_b \to \hat{x}'_b$ is at least $p$, 
and thus, from the definition of the upper numbering, the second lower jump is
at least $p^2-p+1$.  So $\sigma^{r-1}_b \geq p^2 - p + 1 \geq p+1$, which proves
the lemma. 
\end{proof}

%

\subsubsection{The Wild Monodromy Group.}
Finally, we prove Theorem \ref{Tmain}, 
i.e., we show that for every $\gamma \in \Gamma_w$, $\gamma^{p^{n-1}} = 1$. 
Since $\Gamma_w$ acts
faithfully on $\ol{Y}$, it suffices to show that, for an arbitrary point $y \in
\ol{Y}$ and an arbitrary $\gamma \in \Gamma_w$, we have $\gamma^{p^{n-1}}(y) =
y$.

Let us first assume that $y$ is a point of $\ol{Y}$ which lies on an
irreducible component $\ol{Y}_b$ above an \'{e}tale tail $\ol{X}_b$.  By Lemma
\ref{Lfixetale}, the tail $\ol{X}_b$ is
acted on by $\gamma$.
By Lemma \ref{Ltailetale}, $\ol{X}_b$ intersects an inseparable component, and
thus the inertia groups of
$\ol{f}$ above the point of intersection are of order divisible by $p$.  
So the number of irreducible components $N$ lying above $\ol{X}_b$ is divisible
at
most by $p^{n-1}$, not by $p^n$.  Thus $\gamma^{p^{n-1}}(y) \in \ol{Y}_b$.  
Furthermore, if any $\gamma \in \Gamma_w$ acts nontrivially on $\ol{Y}_b$, then
Lemma
\ref{Lfixetalevert} shows that $\gamma$ acts with order $p$ on $\ol{Y}_b$, and
$p$ divides the
number of intersection points of $\ol{Y}_b$ with the rest of $\ol{Y}$.  In this
case, $v_p(N) \leq n-2$, and 
$\gamma^{p^{n-2}}(y) \in \ol{Y}_b$.  In any case, $\gamma^{p^{n-1}}(y) = y$.

Now, assume $y$ is a point of $\ol{Y}_b$, where $\ol{Y}_b$ is an irreducible
component of $\ol{Y}$ lying above an 
inseparable $p^{r'}$-tail $\ol{X}_b$ intersecting a $p^r$-component.  Let
$\epsilon_b = 0$ or $1$ as in Lemma \ref{Lsigmaaction}.  If $\ol{X}_b$ is new, then
by Lemmas \ref{Linsepint}, \ref{Ltailbounds} (i), and \ref{Lsigmaaction}, we have
$\sigma_b - 1 \geq p^{\epsilon_b}$.
Combining this with Corollary \ref{Ctaillimit}, we see that there are fewer than
$p^{r' - \epsilon_b}$ tails of 
type $(r, r')$ with the same $\sigma_b$, and thus $\gamma^{p^{r'-\epsilon_b -
1}}(\ol{X}_b) = \ol{X}_b$.  If $\ol{X}_b$ is not new, then $\gamma$ fixes $\ol{X}_b$ and $\epsilon_b = 0$, so 
we also have $\gamma^{p^{r'-\epsilon_b - 1}}(\ol{X}_b) = \ol{X}_b$.

Suppose a $p$-Sylow subgroup $Q$ of $D_{\ol{Y}_b}/I_{\ol{Y}_b}$ has order $p^a$.
Then there are at most $p^{n-a-r'}$ irreducible components lying above
$\ol{X}_b$.  So
$\gamma' = \gamma^{p^{n-a-\epsilon_b - 1}}$ satisfies $\gamma'(\ol{Y}_b) =
\ol{Y}_b$.  
But, as was remarked before Lemma \ref{Lsigmaaction}, 
we must have $(\gamma')^{p^{a + \epsilon_b}}(y) = y$.  So $\gamma^{p^{n-1}}(y) =
y$, as we wished to prove.

Lastly, assume the remaining case, i.e., that $y$ lies over some point $x$ on an
interior
component $\ol{W}$ of $\ol{X}$.  Suppose first that there exists an \'{e}tale
tail $\ol{X}_b \succ \ol{W}$.  Then, for
any $\gamma \in \Gamma_w$, Corollary \ref{Cfixbottom} shows that $\gamma$ fixes
$\ol{W}$ pointwise.  Then $\gamma$ acts on the fiber above $x$.  
By Lemma \ref{Letaletail}, the component $\ol{W}$ must be inseparable.  Thus
$p^n$ does not
divide the cardinality of the fiber above $x$, so $\gamma^{p^{n-1}}(y) = y$.
Now suppose that there does not exist any \'{e}tale tail lying outward from
$\ol{W}$.  In this case, let $\ol{X}_b \succ
\ol{W}$ be a $p^{r'}$-tail.  Then $\gamma^{p^{r'-1}}(\ol{X}_b) = \ol{X}_b$, and 
$\gamma^{p^{r'-1}}$ acts on the fiber above $x$.  By
Lemma \ref{Ltailetale} and Proposition \ref{Pmonotonic}, the generic inertia
above $\ol{W}$ has order divisible by 
$p^{r'+1}$, 
so $p^{n-r'}$ does not divide the cardinality of this fiber.  Then
$\gamma^{p^{r'-1 +n - r' - 1}}(y) = \gamma^{p^{n-2}}(y)
= y$, finishing the proof of Theorem \ref{Tmain}. \qed

\subsection{Further Restrictions on the Wild Monodromy}
We state here some stronger results than Theorem \ref{Tmain}, which will be useful for Example \ref{Xgoodreduction}.
We maintain the notation of \S\ref{Swild}.

\begin{lemma} \label{Lnopvert}
Let $\ol{W}$ be an inseparable component of $\ol{X}$.  
Suppose there exists $\gamma \in \Gamma_w$ that acts trivially on $\ol{W}$, but non-trivially
above $\ol{W}$.  Then, for any irreducible component $\ol{V}$ of $\ol{Y}$ above $\ol{W}$ with decomposition group
$D_{\ol{V}} \leq G$, we have $m_{D_{\ol{V}}} = 1$.
\end{lemma}

\begin{proof} 
We may assume that the action of $\gamma$ above $\ol{W}$ is of order $p$.
Since the $D_{\ol{V}}$ are conjugate for each $\ol{V}$ above $\ol{W}$, it suffices to prove the lemma for any one $\ol{V}$. 
Let $I_{\ol{V}} \leq D_{\ol{V}}$ be the inertia group of $\ol{V}$. 
We have three cases to consider: \\\\
\emph{Case (1): There exists $\ol{V}$ above $\ol{W}$ on which $\gamma$ acts.} 

We can think of $\gamma$ as a (central) element of $D := D_{\ol{V}}/I_{\ol{V}}$. 
Since $\gamma$ is central in $D$, Corollary \ref{Cnormalp} (ii) shows that $m_D = 1$.  Since $D$
is a quotient of $D_{\ol{V}}$ by a $p$-group, it is easy to see that
$m_{D_{\ol{V}}} = m_D = 1$.  \\\\
\emph{Case (2a): There is no $\ol{V}$ above $\ol{W}$ on which $\gamma$ acts, and $p \nmid |D_{\ol{V}}/I_{\ol{V}}|$.} 

Pick some $\ol{V}$ above $\ol{W}$, and some point $y \in \ol{V}$ that is a smooth point of $\ol{Y}$.
By Proposition \ref{Pspecialram}, the inertia group of $y$ is $I_{\ol{V}}$.
Since $p \nmid |D_{\ol{V}}/I_{\ol{V}}|$, then $D_{\ol{V}} = I_{\ol{V}} \rtimes H$, where
$H$ has prime-to-$p$ order.

Now, $\gamma$ fixes $\ol{W}$, so there is some element $g \in G$ such that
$g(y) = \gamma(y)$.  Since $\Gamma_w$ commutes with $G$, this also implies
$g^a(y) =
\gamma^a(y),$ for all $a \in \ints$. 
We claim that $g$ normalizes $D_{\ol{V}}$.  In fact, even more is true:
$ghg^{-1}(v) = h(v)$ for all $h \in D_{\ol{V}}$ and $v \in
\ol{V}$.  This is because $g hg^{-1}(y) = g h\gamma^{-1}(y) = \gamma^{-1}gh(y)$, and
$\gamma^{-1}g$, being a Galois automorphism of
$\ol{V} \to \ol{W}$ with the fixed point $y$, fixes $\ol{V}$ pointwise.  
So $ghg^{-1}(y) = h(y)$, and $ghg^{-1}$ and $h$, both being elements of $D_{\ol{V}}$
which act the same way on $y$, must act the same way on all $v \in \ol{V}$.
Thus, conjugation by $g$ induces the identity on $D_{\ol{V}}/I_{\ol{V}}$, and $gHg^{-1}$ is
a lift of $H$ in $D_{\ol{V}}$.  Since $H^1(H,I_{\ol{V}}) = 0$ by the
Schur-Zassenhaus theorem, all such lifts differ only by conjugation by an element of $I_{\ol{V}}$, so we have that there is some
$i \in I_V$ such that conjugation by $g$ and conjugation by $i$ act
identically on $H$.  In particular, $g i^{-1}$ centralizes $H$ and normalizes
$D_{\ol{V}}$ and $I_{\ol{V}}$.  By replacing our choice of $g$ with $gi^{-1}$, we may even assume that $g$ centralizes $H$.  

In any case, we know that $\gamma^p(y) = y$, so $g^p(y) = y$.  This implies $g^p \in I_{\ol{V}}$.  
But $g \notin I_{\ol{V}}$, so we must have that $g^p$ generates $I_{\ol{V}}$ (if not, then $g$ and $I_{\ol{V}}$ generate a non-cyclic $p$-group).
Since $g$ centralizes $H$, so does $g^p$.  Then $I_{\ol{V}}$ commutes with $H$,
and $m_{D_{\ol{V}}} = 1$. \\\\
\emph{Case (2b): There is no $\ol{V}$ above $\ol{W}$ on which $\gamma$ acts, and $p \ | \ |D_{\ol{V}}/I_{\ol{V}}|$.}

We show that this case does not arise.
Take $\ol{V}$, $g$, and $y$ as in Case (2a).  As in Case (2a), we have that $g$ normalizes both $D_{\ol{V}}$ and $I_{\ol{V}}$, that $g$ centralizes
$D_{\ol{V}}/I_{\ol{V}}$, and that $g^p \in I_{\ol{V}}$.  Now, consider the group $M \leq G$ generated by $D_{\ol{V}}$ and $g$.  The subgroup
$I_{\ol{V}}$ is normal
in $M$, so let $M' \cong M/I_{\ol{V}}$.  Then the image of $g$ has order $p$ in $M'$ and centralizes the image of $D_{\ol{V}}$ in $M'$.  But the image
of $D_{\ol{V}}$ in $M'$ has a nontrivial element $d$ of order $p$, and so the $p$-subgroup of $M'$ generated by $d$ and the image of $g$ 
is elementary abelian.  This is a contradiction, as $G$ has cyclic $p$-Sylow group. 
\end{proof}

\begin{lemma}\label{Lnopvert2}
As in Lemma \ref{Lnopvert}, let $\ol{W}$ be an inseparable component of $\ol{X}$.  
If there exists $\gamma \in \Gamma_w$ that acts trivially on $\ol{W}$, but non-trivially
above $\ol{W}$, then for each singular point $w$ of $\ol{X}$ on $\ol{W}$, either no \'{e}tale tail lies outward from $w$ or every
\'{e}tale tail lies outward from $w$.
\end{lemma}

\begin{proof}
We first claim that $\sigma^{\eff}_{e} \in \ints[\frac{1}{p}]$
for $e \in E(\mc{G})$ such that $e$ corresponds to $w$ and $s(e)$ corresponds to $\ol{W}$.
Suppose that $e$ is such an edge.  Write $\ol{W}'$ for the component corresponding to $t(e)$, write
$\ol{V}'$ for a component of $\ol{Y}$ above $\ol{W}'$ intersecting $\ol{V}$, and
pick $v \in \ol{V} \cap \ol{V}'$.  Write $I_v$ for the inertia group in $G$
at $v$.  By Lemma \ref{Lnopvert}, we have $m_{D_{\ol{V}}} = 1$. 
Since $I_v \subseteq D_{\ol{V}}$, it follows that $m_{I_v} = 1$.  

Assume that $I_{\ol{V}}$ contains the generic inertia group of $I_{\ol{V}'}$ in
$G$.  Then $\sigma^{\eff}_e$ is defined
using deformation data on $\ol{Z}$, where $\ol{V} \stackrel{\phi}{\to} \ol{Z}
\stackrel{\psi}{\to} \ol{X}$ is such that
$\phi$ is radicial and $\psi$ is tamely ramified at $\phi(v)$.  By Proposition
\ref{Pdefdatadenom}, the invariants of all
of these deformation data above $w$ are integers.  Then the definition of the
effective invariant $\sigma^{\eff}_e$ at $w$ shows that it is in $\ints[\frac{1}{p}]$.  If, instead,
$I_{\ol{V}'}$ contains $I_{\ol{V}}$, the same argument shows that $\sigma^{\eff}_{\ol{e}} \in \ints[\frac{1}{p}]$,
and so $\sigma^{\eff}_e = -\sigma^{\eff}_{\ol{e}} \in \ints[\frac{1}{p}]$.  The claim is proved.

On the other hand, it is not hard to see from the effective local vanishing cycles formula
(\ref{Egenlocvancycles}) and Lemma \ref{Lsigmaeffcompatibility} that 
\begin{equation}\label{Efracpart}
\langle \sigma^{\eff}_{e} \rangle = \langle \sum_{\substack{b \in B_{\text{\'{e}t}} \\ \ol{X}_b \succeq w}} \sigma_b \rangle.
\end{equation}
The right-hand side of (\ref{Efracpart}) is in $\frac{1}{m_G}\ints \subset \frac{1}{p-1} \ints$.  
The vanishing cycles formula (\ref{Evancycles}) shows that it is in $\ints$ iff the right-hand side counts either no \'{e}tale tails or
every \'{e}tale tail.
Since $\ints[\frac{1}{p}] \cap \frac{1}{p-1} \ints = \ints$, this must be the case.
\end{proof}

The following proposition is the main result of this section:
\begin{prop}\label{Ptrivwildmon}
If the branching indices of $f$ are prime to $p$ and $\ol{X}$ has no new \'{e}tale tails, then the wild monodromy $\Gamma_w$ is trivial.
\end{prop}

\begin{proof}
By Proposition \ref{Pnoinsepnonew}, there are no new inseparable tails.  By Proposition
\ref{Pcorrectspec}, there are no inseparable tails at all.
So all tails are primitive \'{e}tale.  Pick $\gamma \in \Gamma_w$.  Then $\gamma$ fixes $\ol{X}$ pointwise (it fixes the interior
components by Corollary \ref{Cfixbottom}, and it fixes the primitive tails because it fixes two points on them).  Now, since every tail is primitive, 
we see that each singular point of $\ol{X}$ has exactly one of the three \'{e}tale tails lying outward from it.  So $\gamma$ acts trivially
above the inseparable components of $\ol{X}$ (Lemma \ref{Lnopvert2}), and thus does not permute the components of $\ol{Y}$ above the 
\'{e}tale tails.  Lastly, by Lemma \ref{Lvertactioncentral}, the action of $\gamma$
above the \'{e}tale tails is trivial.  So $\gamma$ is the identity.
\end{proof}

\subsection{Good reduction}

Recall that $f: Y \to X \cong \proj^1$ is a three-point cover defined over $K$.
The proposition below motivates Proposition \ref{Ptrivwildmon} above:

\begin{prop}\label{Pgoodred} 
If the absolute ramification index $e$ of $K$ is less than $(p-1)/m_G$, and if $f$ has bad reduction, then
$\Gamma_w$ is non-trivial.
\end{prop}
\begin{proof} This is essentially the argument of \cite[\S5.1]{Ra:sp}.
\end{proof}

This has the following consequence:

\begin{corollary} \label{Cgoodreduction}
Let $X = \proj^1_{K}$, where $K$ is a complete discretely valued field of
mixed characteristic $(0,p)$ and absolute ramification index $e$.  Let $f: Y \to X$ be a smooth, geometrically
connected, Galois cover branched only at $\{0,1,\infty\}$ with Galois group $G$, defined over $K$.  Assume that
$G$ has a cyclic $p$-Sylow subgroup, and that $e < (p-1)/m_G$.  
Assume further that there are no new \'{e}tale tails in the stable reduction of $f$.  
Then $f$ has potentially good reduction.
\end{corollary}
\begin{proof} By \cite[Corollaire 4.2.13]{Ra:sp}, the ramification indices of
$f$ are of prime-to-$p$ order.  The Corollary follows from Propositions \ref{Ptrivwildmon} and \ref{Pgoodred}.  
\end{proof}

\begin{example}\label{Xgoodreduction}
We exhibit a family of covers with arbitrarily large cyclic $p$-Sylow subgroup that have good reduction 
by Corollary \ref{Cgoodreduction}.
Fix a prime $p \equiv 1 \pmod{3}$, and $n \geq 1$.
Let $q$ satisfy 
\begin{equation}\label{Egoodred}
q^2 + q + 1 \equiv 0 \pmod{p^n}.
\end{equation}  
We note that, for $n = 1$, this is satisfied whenever $q^3 \equiv 1 \pmod{p}$ and
$q \not \equiv 1 \pmod{p}$.  Since $p \equiv 1 \pmod{3}$, there are solutions of (\ref{Egoodred}) for $n=1$.  By Dirichlet's theorem there are infinitely 
many prime solutions $q$ (and perhaps some higher prime power solutions, too).  Once there is a solution of (\ref{Egoodred}) for $n=1$, Hensel's 
lemma gives solutions for all $n$.  Since these solutions are given by congruence conditions$\pmod{p^n}$, there are also infinitely many prime 
solutions $q$ for any fixed $n$ and $p \equiv 1 \pmod{3}$ (and perhaps some higher prime powers).
The smallest solution for $n > 1$ and $q$ a prime power is $p = 7$, $n = 2$, and $q = 67$. 

Assume $q = \ell^f$ is a prime power satisfying (\ref{Egoodred}). Let $G \cong PGL_3(q)$.  
We know by \cite[Satz 7.3]{Hu:eg} that $G$ has a cyclic $p$-Sylow subgroup of order $p^{v_p(q^2+q+1)} \geq p^n$, with $m_G = 3$.
We will construct a three-point $G$-cover defined over $K_0$ with potentially good reduction to characteristic $p$.

Consider $H := GL_3(q) \rtimes \ints/2$, where the $\ints/2$-action is inverse-transpose.
In \cite[II, Proposition 6.4 and Theorem 6.5]{MM:ig}, a rigid class vector
$(\tilde{C}_0, \tilde{C}_1, \tilde{C}_{\infty})$ is exhibited for 
$H/ \{\pm 1\}$, where $\tilde{C}_0$ has order $2$, $\tilde{C}_1$ has order $4$, and $\tilde{C}_{\infty}$ has order
$(q-1)\ell^a$ for some $a$ (this is because the characteristic polynomial for the elements of $\tilde{C}_{\infty}$ has eigenvalues of 
order $q-1$).  Since $p$ does not divide the order of any of the ramification indices, this triple is rational over $K_0$, so the corresponding
$H/ \{\pm 1\}$-cover is defined over $K_0$.  Thus there is a quotient $G \rtimes \ints/2$-cover $h: Y \to \proj^1$ defined over $K_0$. 

Let $X \to \proj^1 $ be the quotient cover of $h: Y \to \proj^1$ corresponding to the group $G$.  
Then $X \to \proj^1$ is a cyclic cover of degree 2, branched at $0$ and $1$
(this comes from the proof of \cite[II, Proposition 6.4]{MM:ig}).  
This means that $X \cong \proj^1$, and $Y \to X$ is branched at three points (the two points above 
$\infty$, and the unique point above $1$).  So we have constructed a three-point $G$-cover $f: Y \to X \cong \proj^1$ defined over $K_0$, 
and all three branch points have prime-to-$p$ branching index.  Thus, there are three primitive tails (Proposition \ref{Pcorrectspec} and Lemma 
\ref{Letaletail}).

Since $m_G = 3$, each primitive tail $\ol{X}_b$ satisfies $\sigma_b \geq \frac{1}{3}$.  By the vanishing cycles formula (\ref{Evancycles}), 
all $\sigma_b$ are equal to $\frac{1}{3}$ and there are no new (\'{e}tale) tails. 
We note that the absolute ramification index $e$ of $K_0$ is 1, and, since $p > 4$, we have $e < \frac{p-1}{m_G}$.  
We conclude using Corollary \ref{Cgoodreduction}.
\end{example}

\begin{remark}
There are no examples satisfying the hypotheses of Corollary \ref{Cgoodreduction} where $G$ is $p$-solvable.  This is
because there would be a quotient cover with Galois group $\ints/p \rtimes \ints/m_G$ and prime-to-$p$ ramification
indices (\cite[Proposition 2.1]{Ob:fm1}).  This cannot be defined as a Galois cover over a field of ramification index $e <
(p-1)/m_G$ (see \cite[Cor. 1.5]{We:mc}; in fact, one exactly needs $e = (p-1)/m_G$).  
\end{remark}

\appendix

\section{An example of wild monodromy}\label{Awildexample}

Throughout this appendix, let $G = SL_2(251)$, and let $k$ be an algebraically closed field of characteristic $p = 5$.  
Let $R_0 = W(k)$ and $K_0 = \Frac(R_0)$.  Lastly, let $K = K_0(\mu_{5^{\infty}})$ (that is, we adjoin all $5$th-power roots of unity
to $K$), and let $R$ be the valuation ring of $K$.  Note that $G$ has a cyclic $5$-Sylow subgroup of order $5^3 = 125$ and $m_G = 2$.
Our example of a three-point $G$-cover with nontrivial wild monodromy (and such that $5$ does not divide the order of the center of $G$)
depends on intricate calculations from \cite{Ob:fm2}.  We normalize all valuations on $R_0$, $K_0$, or any extensions thereof so that $v(5) = 1$.

\begin{prop}\label{P251cover}
There exists a three-point cover $f: Y \to X = \proj^1_K$, defined over $K$, such that the branching indices of the three branch points are
$e_1$, $e_2$, and $e_3$, with $v_5(e_1) = 0$, $v_5(e_2) = 2$, and $v_5(e_3) = 3$.
\end{prop}

\begin{proof}
We show that such a cover can be defined over $\rats^{ab}$.  Since $\rats^{ab} \hookrightarrow K$, this will prove the proposition.

Let $\alpha = \matrix{1}{1}{0}{1} \in G$.  This has order $251$.
We claim there exists $\beta = \matrix{a}{b}{c}{d} \in G$ satisfying the following properties:
\begin{itemize}
\item The order of $\beta$ is 250.
\item The order of $\alpha\beta$ is 50.
\item The matrices $\alpha$ and $\beta$ generate $SL_2(251)$.
\end{itemize}

To prove the claim, first note that any $GL_2(251)$-conjugacy class in $G$ is determined by the trace of the matrices it contains, 
unless the trace is $\pm 2$.  In particular, the trace of a matrix determines its order if it is not $\pm 2$. 
Let $\tau$ be the trace of the matrices in some conjugacy class of order $250$, and let $\rho$ be the trace of the matrices in some conjugacy
class of order $50$.  Then $\tau$, $\rho$, $2$, and $-2$ are pairwise distinct.
Choose $a$, $b$, $c$, and $d$ in $\FF_{251}$ solving the (clearly solvable) system of equations:
\begin{eqnarray*}
a+d &=& \tau \\
a + c+ d &=& \rho \\
ad - bc &=& 1
\end{eqnarray*}
Since the trace of $\alpha\beta$ is $a + c + d$, these equations ensure that $\beta$ and $\alpha\beta$ have the desired orders.  Let $\ol{\alpha}$ and
$\ol{\beta}$ be the images of $\alpha$ and $\beta$ in $H := PSL_2(251)$.  Since $c \ne 0$,
one checks that $\ol{\beta}$ does not normalize the subgroup generated by $\ol{\alpha}$.  
Then, by \cite[II, Hauptsatz 8.27]{Hu:eg}, we have that $\ol{\alpha}$ and $\ol{\beta}$ generate $H$.  Furthermore, since $\beta$ is diagonalizable
over $GL_2(251)$ and has eigenvalues of order $250$, then $\beta^{125} = -I_2$.  Since $\ol{\alpha}$ and $\ol{\beta}$ generate $H$, and 
$\beta$ generates $\ker(G \to H)$, then $\alpha$ and $\beta$ generate $G$.

Consider the triple $([\alpha], [\beta], [\alpha\beta]^{-1})$ of conjugacy classes of $G$.  By \cite[I, Theorem 5.10 and Remark afterward]{MM:ig}, 
this triple is rigid.  By \cite[I, Theorem 4.8]{MM:ig}, there exists a three-point $G$-cover of $\proj^1$, defined over $\rats^{ab}$, with branching indices
$e_1 = \ord(\alpha) = 251$, $e_3 = \ord(\beta) = 250$, and $e_2 = \ord((\alpha\beta)^{-1}) = 50$.  This completes the proof of the proposition.
\end{proof}

\begin{prop}
If $f: Y \to X = \proj^1_K$ is a cover satisfying the properties of Proposition \ref{P251cover}, then $f$ has nontrivial wild monodromy $\Gamma_w$.
\end{prop}

\begin{proof}
To fix notation, we assume $f$ is branched at $x=0$, $x=1$, and $x=\infty$ of index $e_1$, $e_2$, and $e_3$, respectively, with $v_5(e_1) = 0$, $v_5(e_2) = 2$, and $v_5(e_3) = 3$.
By \cite[Lemma 3.2]{Ob:ac}, the stable reduction of $f$ has both a primitive \'{e}tale tail and a new \'{e}tale tail. 
Construct the \emph{strong auxiliary cover} $f^{str}: Y^{str} \to X$ of $f$ (\cite[\S2.5]{Ob:fm2}).  
This is a four-point $G^{str}$-cover, with $G^{str} \cong \ints/125 \rtimes \ints/2$ such that the
action of $\ints/2$ is faithful.  By \cite[p.\ 22, (3.1), (3.2)]{Ob:fm2}, this cover given by equations
\begin{eqnarray}
z^2 &=& \frac{x-a}{x} \label{E1} \\
y^{125} &=& g(z) := \left(\frac{z+1}{z-1}\right)^r \left(\frac{z + \sqrt{1-a}}{z- \sqrt{1-a}}\right)^s, \label{E2}
\end{eqnarray}
where $r$ and $s$ are integers satisfying $v_5(r) = 0$ and $v_5(s) = 1$.  Replacing $y$ with a prime-to-$p$ power, we can assume $s = 5$.
By \cite[Lemma 3.7]{Ob:fm2}, we have $v(1-a) > 0$ in $K(a)/K$, and then
\cite[Lemmas 3.22 and 3.26(i)]{Ob:fm2} show that we can take $a = 1 - \frac{25}{r^2}$.  In particular, $f^{str}$ is defined over $K$.
By \cite[Proposition 3.31]{Ob:fm2}, the stable model $(f^{str})^{st}: (Y^{str})^{st} \to X^{st}$ of $f^{str}$ has a new inseparable 
$p$-tail $\ol{X}_c$.  We claim that there is an extension $L/K$ such that
$\Gal(L/K)$ acts nontrivially of order $5$ on the stable reduction $\ol{f}^{str}: \ol{Y}^{str} \to \ol{X}$ of $f^{str}$ above $\ol{X}_c$.  
Since $(f^{str})^{st}$ is a quotient of the stable model $(f^{aux})^{st}$ of the (standard) 
auxiliary cover $f^{aux}$ (\cite[\S2.5]{Ob:fm2}), then $\Gal(L/K)$ will act nontrivially of order divisible by $5$
above $\ol{X}_c$ in $(f^{aux})^{st}$ as well.
Lastly, since, above an \'{e}tale neighborhood of $\ol{X}_c$, the stable model $f^{st}$ of $f$ is isomorphic to a set of disconnected copies
of $(f^{aux})^{st}$, the action of $\Gal(L/K)$ on the stable reduction $\ol{f}$ will be nontrivial of order divisible by $5$ above $\ol{X}_c$.  
By assumption, $f$ is defined over $K$, so this will show that $f$ has nontrivial wild monodromy.

It remains to prove the claim.  Let $Z^{str} = Y^{str}/(\ints/125)$, with stable model $(Z^{str})^{st}$ and stable reduction $\ol{Z}^{str}$. 
Then $z$ is a coordinate on $Z^{str}$, and by \cite[Proposition 3.31(iii)]{Ob:fm2}, there is a component of $\ol{Z}^{str}$ above $\ol{X}_c$
containing the specialization $\ol{d}$ of $$z = d := \frac{2 \cdot 5^{7/5}}{r},$$ where we can use any choice
of $5$th root.  Since $\ol{X}_c$ is a $p$-component, there are $25$ points of $\ol{Y}^{str}$ above $\ol{d}$.  
If $g(d)$ is a $5$th power, but not a $25$th power, in $K(d) = K(\sqrt[5]{5})$, 
then if $L = K(d, \sqrt[25]{d})$, the action of $\Gal(L/K(d))$ will permute these 25 points in orbits of order $5$, and we will be done. 
This follows from Lemma \ref{Lpthpower} below.
\end{proof}

\begin{lemma}\label{Lpthpower}
Let $d = \frac{2 \cdot 5^{7/5}}{r}$, where $r$ is a prime-to-$p$ integer and we choose any $p$th root of $5$.  Let $g$ be the rational function in
(\ref{E2}), with $s=5$ and $a = 1 - \frac{25}{r^2}$.  Then $g(d)$ is a $5$th power, but not a $25$th power,
in $K(\sqrt[5]{5})$.
\end{lemma}
 
\begin{proof}
Fix a $5$th root of $5$ in $\ol{K}$, which we will denote by either $\sqrt[5]{5}$ or $5^{1/5}$.  
We first note that $g(d) \in K_0(\sqrt[5]{5})$.  
By \cite[p.\ 38, equation after (3.18)]{Ob:fm2}, we have 
$$g(d) = \pm\left(1 - \frac{8r^3}{75}d^3 - \frac{32r^5}{5^5}d^5\right) + o(5^{9/4}).$$
Upon plugging in $d$ and simplifying, this gives
$$g(d) = \pm(1 - 3\cdot 5^{11/5} - 4 \cdot 5^2) + o(5^{9/4}).$$
Using the binomal theorem, we see that $g(d)$ has a $5$th root $\delta$ in $K_0(\sqrt[5]{5})$, and
$$\delta = \pm(1 - 3\cdot 5^{6/5} - 20) + o(5^{5/4}).$$  We wish to show that $\delta$ is not a $5$th power in $K(\sqrt[5]{5})$.

Now, since $K(\sqrt[5]{5})/K_0(\sqrt[5]{5})$ is abelian, any subextension is Galois.  So if $\delta$ is a $5$th power in $K(\sqrt[5]{5})$,
then taking a $5$th root of $\delta$ must generate a Galois extension of $K_0(\sqrt[5]{5})$.  This is clearly not the case unless 
$\delta$ is already a $5$th power in $K_0(\sqrt[5]{5})$, so it suffices to show that $\delta$ is not a $5$th power in $K_0(\sqrt[5]{5})$.
Since $-1$ is a $5$th power in $K_0$, we may assume that $\delta = 19 + 3 \cdot 5^{6/5} + o(5^{5/4})$.

Suppose that $\epsilon \in K_0(\sqrt[5]{5})$ such that $\epsilon^5 = \delta$, and write 
$$\epsilon = \alpha + \beta \cdot 5^{1/5} + \gamma \cdot 5^{2/5} + \eta \cdot 5^{3/5} + \theta \cdot 5^{4/5},$$
where $\alpha$, $\beta$, $\gamma$, $\eta$, and $\theta$ are in $K_0$.  By Gauss' Lemma, $\epsilon \in R_0[\sqrt[5]{5}]$.  Equating coefficients
of $1$ and $5^{6/5}$ gives the equations
\begin{eqnarray*}
\alpha^4\beta \equiv 3 \pmod{5}\\
\alpha^5 + 5\beta^5 \equiv 19 \pmod{25}
\end{eqnarray*}
The second equation yields $\alpha \equiv 4 \pmod{5}$, and then the first equation yields $\beta \equiv 3 \pmod{5}$.  But then
$\alpha^5 + 5\beta^5 \equiv 24 + 5 \cdot 18 \equiv 14 \not \equiv 19 \pmod{25}$.  So $\epsilon$ cannot exist, and we are done. 
\end{proof}

\begin{remark}\label{Rnosolvable}
The example above is quite complicated, and is not generalizable in any meaningful way (for instance, it depends critically on having $p=5$). 
One hopes for easier examples, but they are difficult to come by.  
For instance, results of \cite[\S4.1]{Ob:fm1} show that no examples of three-point $G$-covers with
nontrivial wild monodromy can exist when $G$ is $p$-solvable and $m_G > 1$.  So if one wants to find an easier example where $p$ does not
divide the order of the center of $G$, one needs to look either at a group that is not $p$-solvable, or at a group of the form
$G \cong H \rtimes \ints/p^n$, where the action of $\ints/p^n$ on $H$ is faithful.
\end{remark}

\begin{acknowledgements}
I would like to thank my thesis advisor, David Harbater, for many helpful
conversations related to this work.  I thank Irene Bouw for help with Example \ref{Xgoodreduction}, and a previous referee 
for useful suggestions.  Thanks also to Farin Rebecca Loeb, for proofreading an earlier version of this
paper.
\end{acknowledgements}

\end{document}